\documentclass{amsart}

\usepackage[T1]{fontenc}
\usepackage[utf8]{inputenc}
\usepackage[english]{babel}

\usepackage{amsthm, amssymb}
\usepackage{amsmath}
\usepackage{xifthen}
\usepackage{enumitem}
\usepackage{color}
\usepackage{url}

\usepackage{fancyhdr}

\newtheorem{theorem}{Theorem}[section]
\newtheorem{proposition}[theorem]{Proposition}
\newtheorem{lemma}[theorem]{Lemma}
\newtheorem{corollary}[theorem]{Corollary}

\theoremstyle{definition}
\newtheorem{definition}{Definition}[section]

\theoremstyle{remark}
\newtheorem{remark}{Remark}[section]

\numberwithin{equation}{section}

\renewcommand\div{\operatorname{div}}

\newcommand{\LRa}{\Leftrightarrow}

\newcommand{\xto}[1]{\xrightarrow{#1}}

\newcommand\Dsur[2]{\frac{d #1}{d #2}}

\newcommand\R{\mathbb{R}}

\newcommand\N{\mathbb{N}}
\newcommand{\Bal}[2]{B_{#1}(#2)}

\newcommand\abs[1]{\left\lvert #1 \right\rvert}
\newcommand{\norm}[1]{\left\| #1 \right\|}
\newcommand{\dprod}[2]{\left\langle #1,#2 \right\rangle}

\newcommand\bpr[1]{\left(#1\right)}
\newcommand\bsq[1]{\left[#1\right]}

\renewcommand{\epsilon}{\varepsilon}
\renewcommand{\emptyset}{{\text{\Large\o}}}

\renewcommand{\phi}{\varphi}

\newcommand\at[1]{\Big\lvert_{#1}}
\newcommand\bigO{\mathcal{O}}
\newcommand{\sub}{\subset}

\newcommand{\D}{\Delta}

\newenvironment{bigcases}{\left\{\begin{aligned}}{\end{aligned}\right.}
\newcommand{\edintertext}[1]{%
  \noalign{%
    \vskip\belowdisplayshortskip
    \vtop{\hsize=\linewidth#1\par
    \expandafter}%
    \expandafter\prevdepth\the\prevdepth
  }%
}

\renewcommand{\a}{\alpha}
\newcommand\Sob[1][k]{\mathcal{H}^{#1}}

\newcommand\Snorm[3][]{\norm{#3}_{\Sob[#2]#1}}

\newcommand{\Lnorm}[3][]{\norm{#3}_{L^{#2}#1}}

\newcommand{\Dg}[1][g]{\D_{#1}}

\newcommand{\intM}[2][M]{\int_{#1} #2\, dv_g}

\newcommand{\inj}{i_g}

\newcommand{\dg}[2][g]{d_{#1}{(#2)}}

\newcommand{\Cct}{C_c^\infty}

\newcommand{\G}[1]{H^{(#1)}}
\newcommand{\Ga}[2][]{H^{\ifthenelse{\isempty{#2}}{}{(#2)}}_{\a\ifthenelse{\isempty{#1}}{}{; #1}}}
\newcommand{\Go}[2][]{H^{\ifthenelse{\isempty{#2}}{}{(#2)}}_{1\ifthenelse{\isempty{#1}}{}{; #1}}}
\newcommand{\Gam}[2][]{\Gamma^{#2}_{\alpha\ifthenelse{\isempty{#1}}{}{,#1}}}

\newcommand{\sqa}{\sqrt{\a}}
\newcommand{\clasG}{\mathfrak{H}}
\newcommand{\Gg}[1][]{\Tilde{G}_{\a\ifthenelse{\isempty{#1}}{}{; #1}}}
\newcommand{\Gga}[1][]{G_{g,\a\ifthenelse{\isempty{#1}}{}{; #1}}}
\newcommand{\expab}[1]{e^{-\sqrt{\a} \dg{#1}}}
\newcommand{\Pea}{\Psi_{\epsilon,\a}}
\newcommand{\expsqa}[2][\epsilon]{e^{-(1-#1)\sqrt{\a}\dg{#2}}}
\newcommand{\Diag}{Diag}

\let\oldtbin\tbinom
\renewcommand{\tbinom}[2]{\oldtbin{#2}{#1}}
\let\oldbin\binom
\renewcommand{\binom}[2]{\oldbin{#2}{#1}}

\begin{document}

\title[Green's function of polyharmonic operators]{The Green's function of polyharmonic operators with diverging coefficients: Construction and sharp asymptotics.}

\author{Lorenzo Carletti}
\date{November 2024}
\thanks{This publication is supported by the French Community of Belgium as part of the funding of a FRIA grant.}
\address{Lorenzo Carletti, Université Libre de Bruxelles, Service d'Analyse, Boulevard du Triomphe - Campus de la Plaine, 1050 Bruxelles, Belgium}
\email{\url{lorenzo.carletti@ulb.be}}
\begin{abstract}
    We show existence, uniqueness and positivity for the Green's function of the operator $(\Delta_g + \alpha)^k$ in a closed Riemannian manifold $(M,g)$, of dimension $n>2k$, $k\in \mathbb{N}$, $k\geq 1$, with Laplace-Beltrami operator $\Delta_g = -\operatorname{div}_g(\nabla \cdot)$, and where $\a >0$.
    We are interested in the case where $\a$ is large : We prove pointwise estimates with explicit dependence on $\a$ for the Green's function and its derivatives. We highlight a region of exponential decay for the Green's function away from the diagonal, for large $\a$.
\end{abstract}
\maketitle

\section{Introduction and statement of the result}
\par Let $(M,g)$ be a smooth connected Riemannian manifold of dimension $n$, compact and without boundary. Let $k\geq 1$ be an integer, we assume that $n>2k$. Let $\a > 0$ be a fixed real number, we consider the elliptic partial differential operator of order $2k$, $(\Dg+\a)^k$ in $M$, where we define $\Dg :=  -\div_g(\nabla \cdot)$ the Laplace-Beltrami operator. In this article we construct the Green's function for $(\Dg + \a)^k$ in $M$, we show uniqueness and positivity, as well as sharp pointwise asymptotics. The main goal is to obtain asymptotics that explicitly depend on $\a$ to understand the behavior of the Green's function when $\a$ is large.

\par The Green's function for polyharmonic problems has been extensively studied in the Euclidean setting, and sharp bounds from above and below have been obtained, see for instance to \cite{DacqMeiSw05,GazGruSw10,Gru21}. 
It is also worth mentioning that there exists an extensive literature for the construction of Green's functions for standard operators of second order on common domains of $\R^n$, in particular in $\R^2,\, \R^3$. See for instance \cite{Duf15} for an undergraduate-level textbook on the matter.
\par Polyharmonic operators on manifolds have been studied because of their connections with the so-called prescribed $Q$-curvature equations. These equations involve a special family of conformally invariant operators, called \emph{GJMS operators} \cite{GJMS92}, for which a Green's function was investigated in \cite{Mic10}. Green's functions for GJMS operators have proven fundamental to obtain existence results for the prescribed $Q$-curvature equations, we refer to \cite{Sch84} for the conformal Laplacian, see also \cite{MazVet} and the references therein for higher-order $Q$-curvature equations.  
Moreover, the operator $(\Dg +\a)^k$ can be seen as a toy-model for the GJMS operator of order $2k$, which, on an Einstein manifold, can be written as a product of $k$ operators of the form $\Dg + c_j$ (see \cite{FefGra12}). 

\par For any $p\geq 1$ and $l \in \N$, let us define the norms
\[  \Snorm[(M)]{l,p}{u}^p := \sum_{m=0}^l \Lnorm[(M)]{p}{\Dg^{m/2} u}^p,
    \]
where we write $\abs{\Dg^{m/2} u} := \begin{cases}
    \abs{\Dg^i u} & \text{if $m=2i$ is even,}\\
    \abs{\nabla \Dg^i u}_g & \text{if $m=2i+1$ is odd}
\end{cases}$.
Let us also define the Sobolev space $\Sob[l,p](M)$ as the closure of $C^\infty(M)$ in $L^p(M)$ with respect to the norm $\Snorm{l,p}{\cdot}$. We write $\Sob(M)$ for the Hilbert space $\Sob(M) = \Sob[k,2](M)$, for $k\geq 1$ integer. 

\par In this paper, we study the operator $(\Dg+\a)^k$, with $\a>0$, in $M$. Our main motivation to consider these specific operators comes from their importance in the study of the optimal constant for the critical Sobolev embeddings in compact manifolds. We refer to \cite{HebVau96} for the case $k=1$, and \cite{Heb03} for the biharmonic case $k=2$, where the operator $(\Dg+\a)^k$, for $\a \to \infty$, naturally appears in a contradiction argument. In a companion article \cite{Car24}, we consider the general case $k\geq 1$. In that case, our proof makes an extensive use of the explicit behavior of the Green's function for the operator $(\Dg+\a)^k$, when $\a \to \infty$. We require a precise description of the Green's function depending explicitly on the coefficient $\a$. The present article thus provides technical results that are crucial in \cite{Car24}. The geometrical context explains that no boundary conditions need to be considered here, the presence of a diverging coefficient being the main challenge.
\par Operators of the form $(\Dg+\a)^k$ also naturally appear in other contexts, as they are simple toy-models to understand the effect of a large parameter in the lower-order terms. We refer to \cite{AdPaYa95,Rey02,Wan95} for instances where concentration phenomena for sequences of solutions to critical non-linear equations are investigated in the case $k=1$. See also \cite{FelHebRob05}, where the energy of sequences of solutions is studied, in the case $k=2$.

\smallskip
\par Observe that, for $\a>0$, the operator $(\Dg+\a)^k$ is coercive, since if $\a \geq 1$,
\begin{align*}
    \Snorm[(M)]{k}{u}^2 &= \sum_{l=0}^{k} \intM{\abs{\Dg^{l/2} u}^2}\\
        &\leq \sum_{l=0}^{k}\tbinom{l}{k}\a^{k-l} \intM{\abs{\Dg^{l/2} u}^2} = \dprod{(\Dg+\a)^k u}{u}_{\Sob[-k],\Sob[k]},
    \end{align*}
and if $\a < 1$
\[  \Snorm[(M)]{k}{u}^2 \leq \tfrac{1}{\a^k} \dprod{(\Dg+\a)^k u}{u}_{\Sob[-k],\Sob[k]}.
    \]
If $\phi \in C^\infty(M)$, the existence and uniqueness of a solution $u \in C^\infty(M)$ to the linear equation
\begin{equation}\label{tmpdgu}
    (\Dg+\a)^k u = \phi \quad \text{on $M$}
\end{equation}
follows from the coercivity of the operator, and from standard elliptic theory. See for instance \cite{GilTrud} for standard existence and regularity results in the case $k=1$, which can be iterated in the case of our operator. This allows us to define a Green's function for this operator.
\begin{definition}[Green's function]\label{def:Green}
    Let $\a> 0$, $k\geq 1$ and $n>2k$, and let $(M,g)$ be a connected compact Riemannian manifold of dimension $n$, without boundary, with Laplace-Beltrami operator $\Dg= -\div_g(\nabla\cdot)$. A Green's function for the operator $(\Dg+\a)^k$ in $M$ is a function $G : M\times M \setminus \{(x,x)\,:\, x \in M\} \to \R$ such that, writing $G_x(y) := G(x,y)$ for all $x \neq y$ in $M$, we have $G_x \in L^1(M)$ for all $x\in M$, and for all $\phi \in C^\infty(M)$ and all $x\in M$,
    \[  \intM{G_x (\Dg+\a)^k \phi} = \phi(x).
        \]
    This is equivalent to saying that $(\Dg+\a)^k G_x = \delta_x$ in the distributional sense on $M$, where $\delta_x$ is the Dirac's delta distribution centered at $x\in M$.
\end{definition}

\par In this work, we are interested in explicit $\a$-dependent estimates for the Green's function and its derivatives, in particular as $\a$ gets large. Our main result is the following Theorem. Note that a closed Riemannian manifold has a positive injectivity radius, $\inj>0$.
\begin{theorem}\label{prop:bvorGga}
    Let $(M,g)$ be a closed Riemannian manifold, of dimension $n\geq 3$, let $k\geq 1$ with $n>2k$, and $\a >0$. The operator $(\Dg + \a)^k$ in $M$ has a unique Green's function $\Gga$, which is positive, symmetric, and is in $C^\infty(M\times M \setminus \{(x,x) \,:\, x\in M\})$. 
    Moreover, there exists $\a_0\geq 1$ such that we have the following :
    \begin{itemize}
        \item There is a constant $C>0$ such that for all $\a \geq \a_0$ and all $x\neq y$ in $M$ with $\sqrt{\a}\dg{x,y} \leq 1$, we have
        \begin{equation}\label{eq:bvorGg}
            \Gga(x,y) = c_{n,k}\, \dg{x,y}^{2k-n} \bpr{1 + \eta_\a(x,y)}
        \end{equation}
        with
        \[  \abs{\eta_\a(x,y)} \leq C \begin{cases}
            \sqa\, \dg{x,y} & n = 2k + 1\\
            \a\, \dg{x,y}^2 \big(1+\abs{\log\sqa\,\dg{x,y}}\big) & n = 2k + 2\\
            \a\, \dg{x,y}^2 & n \geq 2k + 3
            \end{cases},
            \]
        and where $c_{n,k}$ is an explicit positive constant given by \eqref{def:cnk} below.
        \item For all $0< \epsilon < 1$, there is a constant $C_\epsilon >0$ such that for all $\a \geq \a_0$ and all $x,y \in M$ with $\sqrt{\a}\dg{x,y} \geq 1$,
        \[  \Gga(x,y) \leq C_\epsilon \begin{cases}
            \dg{x,y}^{2k-n} \expsqa{x,y} & \text{if } \dg{x,y} < \inj/2\\
            e^{-(1-\epsilon)\sqa \inj/2} & \text{if } \dg{x,y} \geq \inj/2.
        \end{cases}
            \]
    \end{itemize}
\end{theorem}

This Theorem highlights that 
when $\dg{x,y}$ is small in comparison to $1/\sqa$, the Green's function for $(\Dg+\a)^k$ in $M$ behaves to first order as the Green's function for the poly-Laplacian in $\R^n$, $\Dg[\xi]^k$, up to a remainder term on which we prove explicit bounds. On the other hand, when $\dg{x,y}\geq 1/\sqa$, we obtain an exponential decay. In particular, any region of $M$ situated at a fixed distance from a given point $x\in M$ will lie in this regime as $\a$ becomes large. Note also that most of the construction of the Green's function does not rely on the fact that $\a\geq \a_0$. It is only at $\a\to \infty$, however, that exponential estimates at finite distance are of interest. 
\par We point out that in the construction of the Green's function for $(\Dg+\a)^k$ that we give in Theorem \ref{prop:bvorGga}, the main difficulty resides in the presence of the diverging coefficient $\a$, rather than on the terms that could appear in a domain with boundary conditions. In this case, since we work on a closed manifold, there are no boundary conditions to consider. This follows from the geometrical context in which the operator appears. A new contribution of this work is the derivation of estimates on $\Gga$ which are explicit in the lower-order terms' coefficient. This novel consideration is mainly handled in section \ref{sec:remterm}. The methods developed in this article are flexible and could be adapted to construct and obtain sharp estimates on Green's functions for more general elliptic operators with one or more diverging coefficients.

\par Green's function for polyharmonic operators of order $2k$ in $n$-dimensional domains or manifolds, with $n>2k$, and with bounded coefficients, have been known to satisfy estimates of the following type: There exists $C>0$ such that 
\[  \abs{G(x,y)} \leq C \dg{x,y}^{2k-n}
    \]
for all $x\neq y$ (see \cite[Section 4]{GazGruSw10}, \cite{Rob10}). Theorem \ref{prop:bvorGga} improves these estimates for the specific polyharmonic operator $(\Dg+\a)^k$ in $M$, as $\a \to \infty$. The new highlighted dependence in $\a$ for the decay of $\Gga$ draws parallel to the well-known behavior of the Helmholtz kernel for the operator $-\D - \lambda^2$ in $\R^3$ (see for instance \cite{Duf15}). 

\par The article is structured as follows. In Section \ref{sec:DkRn}, we construct a fundamental solution in $\R^n$ for $(\Dg[\xi] + \a)^k$, and prove precise estimates using a modified Giraud's Lemma which is proved in the Appendix \ref{sec:giraud}. Here, we let $\Dg[\xi] = -\sum_{i=1}^n \partial_i^2$ be the non-negative Laplace operator in the Euclidean space. Section \ref{sec:riem} is devoted to the proof of Theorem \ref{prop:bvorGga}. Based on the method of Robert \cite{Rob10}, we iteratively construct an approximation of the Green's function in $M$ preserving the estimates of the Euclidean case. We then conclude the proof of the Theorem thanks to a self-improving argument that allows to estimate the remainder term. This is where we need to handle the presence of the diverging coefficient, which requires a particular attention in this crucial step, see Remark \ref{rem:adiverg}. Finally, we show estimates on the derivatives of the Green's function, in Proposition \ref{prop:bvorDgGga}. 

\begin{remark}[Notational conventions]\
    \begin{enumerate}
        \item We work on a manifold with fixed metric $g$. In the following, unless specified otherwise, all constants only depend on $(M,g)$, $n$, $k$, they are denoted $C$, and their explicit value can vary from line to line, sometimes even in the same line.
        \item Let $f : X \times Y \to \R$ be a function, we will write, for any fixed $x\in X$, $f_x : Y \to \R$ with $f_x(y) := f(x,y)$.
        \item We will write $\Bal{x}{R}$ for the ball of center $x$ and radius $R>0$, either in $M$ or in the Euclidean space $\R^n$, without distinction. We also define the diagonal set $\Diag := \{(x,y)\,:\, x=y\}$ either in $M$ or in $\R^n$, the ambient space will always be clear from context.
    \end{enumerate}
\end{remark}

\section{The Green's function for $(\Dg[\xi] + \a)^k$ in $\R^n$}\label{sec:DkRn}
Throughout this Section, we let $\Dg[\xi] := -\sum_{i=1}^n \partial_i^2$ be the non-negative Laplacian in $\R^n$. We adopt the notation $\Dg[\xi]$ instead of $-\D$ to maintain consistency with the manifold case introduced below.
In this Section, we prove uniqueness and pointwise bounds for the Green's function of the elliptic polyharmonic operator $(\Dg[\xi]+\a)^k$ in the Euclidean space $\R^n$.
\subsection{Green's function of the poly-Laplacian in $\R^n$}
We start by gathering basic results for the fundamental solution of the poly-Laplacian operator $\Dg[\xi]^k$ in $\R^n$.

\par Fix an integer $k\geq 1$, and $n > 2k$, then define
\begin{equation}\label{def:cnk}
    c_{n,k} = \frac{1}{4^{k} \pi^{n/2}(k-1)!}\, \Gamma\bpr{\frac{n-2k}{2}},
\end{equation}
where $\Gamma(t)$ is the well-known Gamma function. This constant $c_{n,k}$ is chosen such that
\begin{equation}\label{def:GDRn}
    \G{k}(x,y) = c_{n,k} \frac{1}{\abs{x-y}^{n-2k}}
\end{equation}
is a fundamental solution for the poly-Laplacian operator $\Dg[\xi]^k$ in $\R^n$, see \cite[Section 2.6]{GazGruSw10}. This means in particular that 
\[  \Dg[\xi]^k \G{k}_x = 0 \quad \text{in the weak sense on $\R^n \setminus \{0\}$}.
    \]

\subsection{Construction and uniqueness.} Fix $\a>0$, the purpose this Section is to show the existence and study the behavior of the Green's function for $(\Dg[\xi]+\a)^k$ in $\R^n$. In particular, we are interested in its dependence on the coefficient $\a >0$. We start by observing that the Green's functions for different values of $\a$ are related by a simple scaling property. Then, we obtain an exact expression for the Green's function of the operator $\Dg[\xi]+1$. In a second step, we will use properties of the convolutions of distributions to retrieve expressions for the polyharmonic operators $(\Dg[\xi]+\a)^k$, $k\geq 1$. Finally, we use a modified version of Giraud's Lemma, proved in \ref{prop:expGir}, to obtain sharp pointwise bounds on the Green's function.
\begin{definition}
    Fix $k\geq 1$, $n> 2k$ and $\a >0$, we say that $\Ga{k}(x,y)$ is a fundamental solution for the polyharmonic operator $(\Dg[\xi]+\a)^k$ in $\R^n$ if for all $x\in \R^n$, $\Ga[x]{k}\in L^1_{loc}(\R^n)$ and
    \[  \int_{\R^n} \Ga{k}(x,y) (\Dg[\xi] + \a)^k \phi(y)\, dy = \phi(x) \quad \text{for all $\phi \in \Cct(\R^n)$.}
        \]
\end{definition}
\begin{remark}\label{rem:GkaGk1}
    It is straightforward to compute that $\Go{k}(x,y)$ is a fundamental solution for the operator $(\Dg[\xi]+1)^k$ in $\R^n$ if and only if
    \begin{equation}\label{eq:GkaGk1}  
        \Ga{k}(x,y) := \a^{\frac{n-2k}{2}} \Go{k}(\sqa x,\sqa y)
        \end{equation}
    is a fundamental solution for the operator $(\Dg[\xi]+\a)^k$ in $\R^n$. 
\end{remark}
With this observation, we only study the Green's function for $(\Dg[\xi]+1)^k$ in a first step, and then retrieve the general case $\a>0$ using relation \eqref{eq:GkaGk1}.
\par Recall the definition of the Bessel function of the second kind of order $\nu > 0$, $K_\nu(r)$, which is singular at the origin, and solution to the second order ordinary differential equation
\[  u''(r) + \frac{1}{r} u'(r) - \bpr{1 + \frac{\nu^2}{r^2}} u(r) = 0 \qquad \text{on $\R^+ \setminus \{0\}$}.
    \]
These are well-known functions with explicit behavior (see \cite{AbrSteg}).

\begin{proposition}\label{prop:GreenD1}
    Fix $n\geq 3$, then
    \[  \Go{}(x,y) := (2\pi)^{-\frac{n}{2}} \abs{x-y}^{-\frac{n-2}{2}} K_{\frac{n-2}{2}}(\abs{x-y})
        \]
    is a fundamental solution for the operator $(\Dg[\xi]+1)$ in $\R^n$.
\end{proposition}
\begin{proof}
    Start by observing that, thanks to the asymptotics for $K_\nu$ found in \cite{AbrSteg}, we have the following :
    \begin{itemize}
        \item When $\abs{x-y} \ll 1$,
        \begin{equation}\label{eq:bvorGa1}
        \begin{aligned}
            \Go{}(x,y) &= \frac{\pi^{-\frac{n}{2}}}{4} \Gamma\bpr{\frac{n-2}{2}} \abs{x-y}^{-(n-2)}(1+o(1))\\
                &= \frac{1}{(n-2)\omega_{n-1}} \abs{x-y}^{2-n}(1+o(1))\\
            \frac{\partial}{\partial y_i} \Go{}(x,y) &= \frac{1}{\omega_{n-1}} \frac{(x_i - y_i)}{\abs{x-y}^n}(1+o(1))\\
            \LRa \quad \abs{\nabla \Go{}}(x,y) &= \frac{1}{\omega_{n-1}} \abs{x-y}^{1-n} (1+o(1));
        \end{aligned}
        \end{equation}
        \item When $\abs{x-y} \gg 1$,
        \begin{equation}\label{eq:bvorGa2}
        \begin{aligned}
            \Go{}(x,y) &= \frac{(2\pi)^{-\frac{n-1}{2}}}{2} \abs{x-y}^{-\frac{n-1}{2}} e^{-\abs{x-y}}(1+o(\abs{x-y}^{-1}))\\
            \abs{\nabla \Go{}}(x,y) &= \frac{(2\pi)^{-\frac{n-1}{2}}}{2} \abs{x-y}^{-\frac{n-1}{2}} e^{-\abs{x-y}}(1+o(\abs{x-y}^{-1})).
        \end{aligned}
        \end{equation}
    \end{itemize}
    We now show that
    \begin{equation}\label{tmp:bes0}
        (\Dg[\xi]+1)\Go[x]{}(y) = 0 \qquad \text{for all $y \neq x$.}
    \end{equation}
    By the expression of $\Go{}$, we write $r = \abs{x-y}$ and define
    \[  H(r) := \Go{}(x,y) = (2\pi)^{-\frac{n}{2}} r^{-\frac{n-2}{2}} K_{\frac{n-2}{2}}(r).
        \]
    Using the expression of the Laplacian in spherical coordinates on $\R^n$, \eqref{tmp:bes0} re-writes as
    \[  \Dsur{^2}{r^2} H(r) + \frac{n-1}{r} \Dsur{}{r}H(r) - H(r) = 0.
        \]
    Now $H(r)$ satisfies this last equation for $r>0$, by the definition of the Bessel function of the second kind $K_{\frac{n-2}{2}}(r)$. Thus, we conclude that $\Go{}$ solves \eqref{tmp:bes0} for all $x\neq y$.
    \par For the second part of the proof, take $\phi \in \Cct(\R^n)$, we show that for all $x \in \R^n$
    \[  \int_{\R^n} \Go[x]{}(y) (\Dg[\xi] + 1) \phi(y) \,dy = \phi(x).
        \]
    We have
    \begin{equation}\label{eq:decGabc}
        \begin{aligned}
            \int_{\R^n} (\Dg[\xi] \phi + \phi) \Go[x]{} \, dy &= \lim_{\delta \to 0} \int_{\Bal{0}{\delta}^c} (\Dg[\xi] \phi + \phi) \Go[x]{}\, dy\\
            &= \lim_{\delta \to 0} \big[ \,a(\delta) + b(\delta) + c(\delta)\,\big] ,
        \end{aligned}
    \end{equation}
    where 
    \[  a(\delta) := \int_{\Bal{x}{\delta}^c} (\Dg[\xi] \Go[x]{} + \Go[x]{})\, \phi \,dy = 0
        \]
    since $\Go[x]{}$ satisfies \eqref{tmp:bes0} on $\R^n \setminus\{x\}$, and with \eqref{eq:bvorGa1},
    \begin{equation}\label{eq:abcexpr}
        \begin{aligned}
            b(\delta) &:= \int_{\partial \Bal{x}{\delta}} \Go[x]{} \partial_\nu \phi\, d\sigma(y) \sim \delta^{-(n-2)}\delta^{n-1} = \bigO(\delta),\\
            c(\delta) &:= -\int_{\partial \Bal{x}{\delta}} \partial_\nu \Go[x]{} \phi \, d\sigma(y) = \frac{\pi^{-\frac{n}{2}}}{2} \Gamma\bpr{\frac{n}{2}} \omega_{n-1} \,\phi(x) + o(1)
        \end{aligned}
        \end{equation}
    as $\delta\to 0$. Using the expression for the surface area of the sphere, we get
    \[  \int_{\R^n} (\Dg[\xi] \phi + \phi) \Go{} \, dy = \phi(x).
        \]
    We can conclude that $\Go{}(x,y)$ is a fundamental solution for the operator $(\Dg[\xi]+1)$ on $\R^n$. 
\end{proof}
\begin{remark}\label{rem:Go1asG}
    We additionally observe that, for $ \abs{x-y} \ll 1$, the Green's function $\Go{}$ and its gradient are equal to first order to the standard Green's function for the Laplacian in $\R^n$ and its gradient, respectively. Moreover, $\Go{} \in C^\infty(\R^n\times \R^n \setminus \Diag)$ only depends on $\abs{x-y}$.
\end{remark}

The following result is technical and establishes improved bounds for fundamental solutions.
\begin{lemma}\label{prop:decL1}
    Let $u \in C^{2k}(\R^n \setminus \{0\})$ be a function satisfying $(\Dg[\xi]+1)^k u = 0$ on $\R^n \setminus \{0\}$, and such that there exist $C>0$, $\rho \in \R$, with 
    \begin{equation*}
        \abs{u(x)} \leq \begin{cases}
            C \abs{x}^{2k-n} & \text{if } 0<\abs{x} \leq 1\\
            C \abs{x}^\rho e^{-\abs{x}} & \text{if } \abs{x} \geq 1
        \end{cases}.
    \end{equation*}
    Then for $l = 0, \ldots 2k$, there is $C_l >0$ such that
    \begin{equation*}
        \abs{\nabla^l u(x)} \leq C_l \begin{cases}
            \abs{x}^{-(n-2k+l)} & 0< \abs{x} \leq 1\\
            \abs{x}^\rho e^{-\abs{x}} & \abs{x} \geq 1
        \end{cases}.
    \end{equation*}
\end{lemma}
\begin{proof}
    Let $x\neq 0$ be fixed. Notice that there is $C>0$ such that $t^\rho e^{-t} \leq C t^{2k-n}$ for all $t\geq 1$, so that 
    \begin{equation}\label{tmp:bndu}
        \abs{u(x)} \leq C \abs{x}^{2k-n} \qquad \forall~x\neq 0.
    \end{equation} 
    When $\abs{x} \leq 2$, define $v(y) = u(\abs{x} y)$ on a ball $\Bal{\frac{x}{\abs{x}}}{1/2}\not\ni 0$, and write $\lambda = \abs{x}^2 \neq 0$. We compute
    \begin{equation*}  
        (\Dg[\xi] + \lambda)^k v(y) = \abs{x}^{2k} \bpr{(\Dg[\xi]+1)^k u}(\abs{x} y) = 0 \qquad \forall\, y \in \Bal{\frac{x}{\abs{x}}}{1/2}.
        \end{equation*}
    By standard elliptic theory, since $\lambda \leq 4$, there is $C>0$ independent of $x$ such that $v\in C^{2k}(\Bal{\frac{x}{\abs{x}}}{1/4})$, and for all $y \in \Bal{\frac{x}{\abs{x}}}{1/4}$
    \begin{equation}\label{tmp:ineq}
    \begin{aligned}
        \abs{\nabla^l v(y)} &\leq C \Big( \Lnorm[(\Bal{\frac{x}{\abs{x}}}{1/2})]{\infty}{(\Dg[\xi] + \lambda)^k v} + \Lnorm[(\Bal{\frac{x}{\abs{x}}}{1/2})]{\infty}{v}\Big)\\
        &= C \sup_{z \in \Bal{\frac{x}{\abs{x}}}{1/2}} \abs{u(\abs{x} z)} \leq \sup_{z \in \Bal{\frac{x}{\abs{x}}}{1/2}} \frac{C}{\abs{\abs{x}z}^{n-2k}} \leq \frac{C}{\abs{x}^{n-2k}}
    \end{aligned}
    \end{equation}
    using \eqref{tmp:bndu}, and since $\abs{z} \geq 1/2$ for all $z\in \Bal{\frac{x}{\abs{x}}}{1/2}$. Now,
    \[  \abs{\nabla^l v(y)} = \abs{x}^l \abs{\nabla^l u}(\abs{x} y),
        \] 
    we evaluate inequality \eqref{tmp:ineq} at $y = \frac{x}{\abs{x}} \in \Bal{\frac{x}{\abs{x}}}{1/4}$ to obtain
    \[  \abs{\nabla^l u}(x) \leq \frac{C}{\abs{x}^{n-2k+l}} \qquad \forall\, \abs{x} \leq 2.
        \]
    \par On the other hand, when $\abs{x} > 2$, we use elliptic theory for $u$ on a ball $\Bal{x}{1} \subset \R^n \setminus \Bal{0}{1}$. This shows that there is $C>0$ independent of $x$ such that for all $y \in \Bal{x}{1/2}$,
    \begin{align*}  \abs{\nabla^l u}(y) &\leq C \Big( \Lnorm[\Bal{x}{1}]{\infty}{(\Dg[\xi] + \lambda)^k u} + \Lnorm[(\Bal{x}{1})]{\infty}{v}\Big) \leq C \sup_{z \in \Bal{x}{1}} \abs{z}^\rho e^{-\abs{z}},
        \end{align*}
    where this last inequality follows from the assumption on $u$.
    Observe that, when $\abs{x} \geq 2$, $\frac{\abs{x}}{2} \leq \abs{y} \leq \frac{3\abs{x}}{2}$ and $\abs{y} \geq \abs{x} -1$ for all $y \in \Bal{x}{1}$, so that no matter if $\rho$ is positive or negative,
    \[  \abs{\nabla^l u}(y) \leq C \abs{x}^\rho e^{-\abs{x}}.
        \]
    Evaluating this inequality at $y = x \in \Bal{x}{\epsilon}$ gives the result for $\abs{x}>2$.
    \par For the intermediate values $1 \leq \abs{x} \leq 2$, the two regimes coincide, up to a constant. The first part of the proof gives 
    \[  \abs{\nabla^l u}(x) \leq \frac{C}{\abs{x}^{n-2k+l}}.
        \]
    But now for $l = 0,\ldots 2k$, we have $C_1, C_2$ independent of $x$ such that, if $1 \leq \abs{x} \leq 2$,
    \[\begin{aligned}  
        \abs{x}^{-(n-2k+l)} &\leq C_1 \\
        C_1 \leq \abs{x}^\rho e^{-\abs{x}} &\leq C_2 
    \end{aligned},
    \]
    and we conclude.
\end{proof}

\begin{definition}\label{def:clasG}
    Fix $k\geq 1$ and $n>2k$, we define the space $\clasG_k$, of all functions $u \in C^{2k}(\R^n \times \R^n \setminus \Diag)$ such that the following holds.
    \begin{itemize}
        \item There exists $C>0$ such that $\abs{u(x,y)} \leq C\abs{x-y}^{-(n-2k)}$ when $\abs{x-y}\leq 1$;
        \item For all $p\geq 1$, there is $C_p >0$ such that for $l = 0, \ldots 2k$,
        \[  \abs{\nabla^l u(x,y)} \leq C_p \abs{x-y}^{-p} \quad \text{when $\abs{x-y}\geq 1$}.
        \]
    \end{itemize}
\end{definition}

\begin{lemma}\label{prop:uniqL1}
    There is a unique fundamental solution of $(\Dg[\xi] + 1)^k$ in $\R^n$ in the class $\clasG_k$. 
\end{lemma}
\begin{remark}
    Note that with this result, the function $\Go{}: \R^n \times \R^n \setminus \Diag \to \R$ defined in Proposition \ref{prop:GreenD1} is the unique fundamental solution of the operator $\Dg[\xi] + 1$ in $\R^n$ in the class $\clasG_1$.
\end{remark}
\begin{proof}
    Let $x\in \R^n$, and $H, \Tilde{H} \in \clasG_k$ be two fundamental solutions for $(\Dg[\xi] + 1)^k$ in $\R^n$. Define $f_x(y) := H(x,y) - \Tilde{H}(x,y)$. Then $f_x\in C^{2k}(\R^n\setminus\{x\})$ with 
    \begin{enumerate}
        \item $\displaystyle\abs{f_x(y)} \leq C \abs{x-y}^{2k-n}$ when $\abs{x-y}\leq 1$;
        \item $\displaystyle\abs{\nabla^l f_x} \leq C_p \abs{x-y}^{-p}$ for all $p$, when $\abs{x-y}\geq 1$, and for $l=0,\ldots 2k$;
        \item $f_x$ satisfies $(\Dg[\xi]+1)^k f_x = 0$ in the weak sense on $\R^n$.
    \end{enumerate}
    \par We start by proving that the singularity of $f_x$ at $x$ is removable. Note that $f_x\in L^p(\Bal{x}{1})$ for all $1\leq p < \frac{n}{n-2k}$. Elliptic theory gives that $f_x \in \Sob[2k,p](\Bal{x}{1/2})$ and there is $C>0$ independent of $x$ such that
    \begin{align*}
        \Snorm[(\Bal{x}{1/2})]{2k,p}{f_x} &\leq C \bpr{\Lnorm[(\Bal{x}{1})]{p}{f_x} + \Lnorm[(\Bal{x}{1})]{p}{(\Dg[\xi]+1)^k f_x}}\\
            &= C \Lnorm[(\Bal{x}{1})]{p}{f_x}.
    \end{align*}
    Iterating, and by elliptic theory, we similarly find that $f_x \in \Sob[l,p](\Bal{x}{1/2})$ for any $l\geq 0$. By Sobolev embeddings, using a big enough $l$ in the previous argument, then $f_x \in C^{0,\delta}(\Bal{x}{1/4})$ for some $\delta>0$. 
    We conclude that $f_x$ has no singularity at 0, and $f_x\in C^{2k}(\R^n)$ satisfies 
    \begin{equation}\label{eq:tempd1}
        (\Dg[\xi] + 1)^k f_x = 0 \qquad \text{ on $\R^n$ in the classical sense.}
    \end{equation} 
    \par Using the decay of $f_x$ at infinity, we have $f_x \in L^p(\R^n)$ for all $p \geq 1$. Fix $R>0$, and take $\chi_R$ a cutoff function supported in $\Bal{x}{R}$ such that $\chi_R \equiv 1$ on $\Bal{x}{R/2}$. We compute
    \begin{align*}
        (\Dg[\xi]+1)^k (\chi_R f_x) &= \sum_{l=0}^k \tbinom{l}{k} \bpr{\chi_R (\Dg[\xi])^l f_x + \bigO\bpr{\sum_{m=1}^{2l} \abs{\nabla^m\chi_R} \abs{\nabla^{2l-m}f_x}}} \\
            &= \bigO\bpr{\sum_{l=0}^{k} \sum_{m=1}^{2l} \abs{\nabla^m \chi_R} \abs{\nabla^{2l-m} f_x}}
    \end{align*}
    with \eqref{eq:tempd1}. Testing this equation against $\chi_R f_x \in \Cct(\R^n)$, we have by integration by parts on the left-hand side, for all $p>1$ and $R>2$,
    \begin{align*}
        \int_{\R^n} \sum_{l=0}^k \tbinom{l}{k} \abs{(\Dg[\xi])^{l/2}(\chi_R f_x)}^2 dy &\leq C\sum_{l=0}^{k} \sum_{m=1}^{2l} \int_{\R^n} \abs{\nabla^m \chi_R} \abs{\nabla^{2l-m} f_x}\chi_R \abs{f_x} \, dy \\
            &\leq C_p \int_{\Bal{x}{R} \setminus\Bal{x}{R/2}} \abs{x-y}^{-p} dy
    \end{align*}
    where the latter follows from the decay of $f_x$ at infinity. Now $\abs{x-y}^{-p} \in L^1(\R^n \setminus \Bal{x}{1})$ for $p>n$, and thus the right-hand side vanishes as $R\to \infty$ provided we choose a fixed $p>n$. On the other hand since all the terms in the left-hand side are positive,
    \begin{align*}
        \Snorm[(\R^n)]{k}{f_x} &\leq \int_{\R^n} \sum_{l=0}^k \tbinom{l}{k} \abs{\Dg[\xi]^{l/2}f_x}^2 dy \\
            &\leq \liminf_{R \to \infty} \int_{\R^n} \sum_{l=0}^k \tbinom{l}{k} \abs{\Dg[\xi]^{l/2}(\chi_R f_x)}^2 dy = 0
        \end{align*}
    by Fatou's Lemma. Thus $f =0$ everywhere in $\R^n$.
\end{proof}

\begin{remark}\label{rem:uniqLap}
    With the same strategy of proof as in the previous Lemma \ref{prop:uniqL1}, we can show the following. The function $\G{k}(x,y)$ defined in \eqref{def:GDRn} is the unique fundamental solution for the poly-Laplacian operator $\Dg[\xi]^k$ in $\R^n$ in the class of functions $u \in C^{2k}(\R^n \times \R^n \setminus \Diag)$ such that there exists $C>0$ and 
    \[  \abs{u(x,y)} \leq C \abs{x-y}^{2k-n} \quad \forall\, x\neq y.
        \]
\end{remark}

\begin{remark}\label{rem:smthGr}
    A fundamental solution $h$ for $(\Dg[\xi]+1)^k$ is always smooth away from its singularity: Let $h$ be a distribution that satisfies
    \[  (\Dg[\xi]+1)^k h_x = 0 \qquad \text{weakly on $\R^n \setminus\{x\}$}.
        \]
    Now for all $\Omega \subset \R^n$ such that $x \not\in \overline{\Omega}$, $h$ satisfies
    \[  (\Dg[\xi] + 1)^k h_x = 0 \qquad \text{weakly on $\overline{\Omega}$}
        \]
    and by elliptic theory we can conclude $h_x \in C^{\infty}(U)$ for an open set $U \sub\sub \Omega$. This gives in turn $h_x \in C^\infty(\R^n \setminus \{x\})$.
\end{remark}

We have everything we need to construct the Green's function of $(\Dg[\xi] + \a)^k$ in $\R^n$ and describe its exact behavior. Define $\Go{1} := \Go{}$ and for $k\geq 1$, iteratively
\begin{equation}\label{def:Gok}
    \Go{k+1}(x,y) := \Go{k} * \Go{1} (x,y) = \int_{\R^n} \Go{k}(x,z) \Go{1}(z,y)\, dz,
    \end{equation}
which is well-defined provided $2k+2 < n$, as easily seen by iteratively applying Giraud's Lemma (see Lemma \ref{prop:expGir} below).
\begin{theorem}\label{prop:uniqGa}
    Fix $k\geq 1$, $n> 2k$, and $\a > 0$, then 
    \[  \Ga{k}(x,y) := \a^{\frac{n-2k}{2}} \Go{k}(\sqa x, \sqa y)
        \] 
    is the unique Green's function for $(\Dg[\xi] + \a)^k$ in $\R^n$ in the class $\clasG_k$ defined in \ref{def:clasG}, where $\Go{k}$ is as defined in \eqref{def:Gok}. Moreover, there exists $C>0$ independent of $\a>0$ such that for all $x\neq y$,
    \begin{equation}\label{eq:Goest}
        \Ga{k}(x,y) \leq \begin{cases}
        C \abs{x-y}^{2k-n} & \text{when } \sqa \abs{x-y} \leq 1\\
        C \a^{k\frac{n-3}{4}} \abs{x-y}^{\frac{(k-2)n + k}{2}} e^{-\sqa \abs{x-y}} & \text{when } \sqa \abs{x-y} \geq 1
    \end{cases}.
        \end{equation}
    Finally, the Green's function is radial, $\Ga{k}(x,y)$ only depends on $\abs{x-y}$.
\end{theorem}
\begin{proof}
    We begin by showing that $\Ga{k}$ is a fundamental solution for the operator $(\Dg[\xi]+\a)^k$ in $\R^n$. We then get explicit $\a$-dependent estimates for $\Ga{k}$, and deduce its uniqueness.
    \par Since $n>2k$, $\Go{l}$ is defined for all $l = 1,\ldots k$. We prove the first statement by induction. First, $\Go{1}(x,y) = \Go{}(x,y)$ is a Green's function for $\Dg[\xi]+1$ in $\R^n$ as we showed in Proposition \ref{prop:GreenD1}. Assume that we have proven that $\Go{l}$ is a Green's function for some $1\leq l \leq k-1$. Then let $\phi \in \Cct(\R^n)$, we have
    \begin{multline}\label{eq:gkisgreen}
        \int_{\R^n} \Go[x]{l+1}(y) (\Dg[\xi] + 1)^{l+1} \phi(y) dy \\
        \begin{aligned}
            &= \int_{\R^n} \bpr{\int_{\R^n} \Go{l}(x,z) \Go{1}(z,y) dz} (\Dg[\xi]+1)^{l+1} \phi(y) dy\\
            &= \int_{\R^n} \Go[x]{l}(z) \bpr{\int_{\R^n} \Go[z]{1}(y) (\Dg[\xi]+1)\bsq{(\Dg[\xi]+1)^{l} \phi}(y) dy} dz\\
            &= \int_{\R^n} \Go[x]{l}(z) (\Dg[\xi] + 1)^{l}\phi(z) dz\\
            &= \phi(x),
    \end{aligned}\end{multline}
    where the last line is the induction assumption. Now using \eqref{eq:GkaGk1}, we obtain that $\Ga{k}(x,y) = \a^{\frac{n-2k}{2}}\Go{k}(\sqa x, \sqa y)$ is a Green's function for $(\Dg[\xi]+\a)^k$ in $\R^n$.
    \par To prove pointwise estimates on $\Go{k}(x,y)$, we use an exponential version of the so-called \emph{Giraud's Lemma}, whose standard proof can be found in \cite{Gir29}. We prove this result in appendix \ref{sec:giraud} in the generalized setting of a manifold, following a similar reasoning. With the behavior of $\Go{1}$ in \eqref{eq:bvorGa1}, \eqref{eq:bvorGa2}, and Lemma \ref{prop:expGir}, we get iteratively for $l=1,\ldots k$,
    \[  \Go{l}(u,v) \leq \begin{cases}
        C \abs{u-v}^{2l - n} & \text{when } \abs{u-v} \leq 1\\
        C \abs{u-v}^{-l\frac{n-1}{2} + (l-1) n} e^{-\abs{u-v}} & \text{when } \abs{u-v} \geq 1.
    \end{cases}
        \]
    We then observe that $\Go{k} \in \clasG_k$ using Lemma \ref{prop:decL1}, it is thus the only Green's function in this class by Lemma \ref{prop:uniqL1}. Similarly, by relation \eqref{eq:GkaGk1}, we conclude that $\Ga{k}$ is the unique Green's function for $(\Dg[\xi] + \a)^k$ in $\R^n$ in the class $\clasG_k$. The previous estimates now become
    \[  \Ga{k}(x,y) \leq \begin{cases}
        C\abs{x-y}^{2k-n} & \sqa \abs{x-y} \leq 1\\
        C\a^{k \frac{n-3}{4}} \abs{x-y}^{\frac{n(k-2)+k}{2}} e^{-\sqa \abs{x-y}} & \sqa \abs{x-y}\geq 1.
    \end{cases}
        \]
    \par Finally, this Green's function inherits its symmetry from fact that $\Go{1}$ only depends on $\abs{x-y}$ and that the convolution of two radial functions is itself radial.
\end{proof}
Note that the exponent $\frac{(k-2)n + k}{2}$ in \eqref{eq:Goest} becomes positive for $k\geq 2$.

\subsection{Refined asymptotics.} We now prove more precise pointwise estimates on the Green's function $\Ga{k}$ and its derivatives. When $\sqa \abs{x-y}$ is small, we show that $\Ga{k}$ and its first $2k-1$ derivatives are equal to first order to the standard Green's function for the poly-Laplacian in $\R^n$ and its derivatives, respectively.

\begin{proposition}\label{prop:bvorDGak}
    Fix $k\geq 1$, $n>2k$ and $\a > 0$. Then $\Ga{k} \in C^\infty(\R^n \times \R^n\setminus\Diag)$, where $\Ga{k}$ is defined in Theorem \ref{prop:uniqGa}, and for all $l = 0, \ldots 2k$, there exists $C_l >0$ independent of $\a$ such that for all $x\neq y$ in $\R^n$,
    \[  \abs{\nabla^l \Ga[x]{k}(y)} \leq \begin{cases}
        C_l \abs{x-y}^{-(n-2k+l)} & \sqa \abs{x-y} \leq 1\\
        C_l\, \a^{k\frac{n-3}{4} + \frac{l}{2}} \abs{x-y}^{\frac{(k-2)n + k}{2}} e^{-\sqa \abs{x-y}} & \sqa \abs{x-y} \geq 1.
    \end{cases}
        \]
\end{proposition}

\begin{proof}
    The first part $\Ga{k} \in C^\infty(\R^n \times \R^n\setminus\Diag)$ comes from the Remark \ref{rem:smthGr}. The estimates for $l\geq 1$ are then a direct consequence of Lemma \ref{prop:decL1}. For $l=1,\ldots 2k$ we have
    \begin{align*}
        \abs{\nabla^l \Ga[x]{k}(y)} =& \,\a^{\frac{n-2k}{2}} \a^{\frac{l}{2}} \bpr{\nabla^l \Go[\sqa x]{k}}(\sqa y)\\
            \leq& \,C_l \a^{\frac{n-2k}{2}+\frac{l}{2}} \begin{cases}
                 \big(\sqa \abs{x-y}\big)^{-(n-2k+l)} & \sqrt{\a}\abs{x-y} \leq 1\\
                \big(\sqa \abs{x-y}\big)^{\frac{(k-2)n + k}{2}} e^{-\sqa\abs{x-y}} & \sqrt{\a} \abs{x-y} \geq 1
            \end{cases}\\
            &= C_l\begin{cases}
                \abs{x-y}^{-(n-2k+l)} & \sqrt{\a}\abs{x-y} \leq 1\\
                \a^{k\frac{n-3}{4} + \frac{l}{2}} \abs{x-y}^{\frac{(k-2)n + k}{2}} e^{-\sqa\abs{x-y}} & \sqrt{\a}\abs{x-y} \geq 1
            \end{cases}.
    \end{align*}
\end{proof}

We now prove precise estimates for the behavior of $\Ga{k}$, when $\sqa \abs{x-y}$ is small. To simplify the notation, define
\begin{equation}\label{def:eta}  
    \eta(t) = \begin{cases}
    t & \text{when $n=2k+1$}\\
    t^2 \big(1+\abs{\log t}\big) & \text{when $n=2k+2$}\\
    t^2 & \text{when $n\geq 2k+3$}
    \end{cases}, \qquad \text{for $0<t \leq 1$.}
\end{equation}
\begin{proposition}\label{prop:excbvorGa}
    Fix $k \geq 1$, $n>2k$, $\a >0$, and let $\Ga{k}$ be the unique Green's function in $\clasG_k$ for the operator $(\Dg[\xi]+\a)^k$ in $\R^n$. Then, when $\sqa \abs{x-y} \leq 1$,
    \[  \Ga{k}(x,y) = c_{n,k}\abs{x-y}^{2k-n}\bpr{1+\bigO\big(\eta(\sqa \abs{x-y})\big)},
        \]
    where $\eta$ is defined in \eqref{def:eta} and $c_{n,k}$ is the constant in \eqref{def:cnk}.
\end{proposition}
\begin{remark}
    Here and in the following, the notation $f(x,y) = \bigO(u(x,y))$, for a positive function $u$, is used to mean that there is a constant $C>0$, independent of $\a$, such that for all $x,y$,
    \[  \frac{\abs{f(x,y)}}{u(x,y)} \leq C.
        \]
\end{remark}
\begin{proof}
    We begin by defining $ R_\a := \Ga{k} - \G{k}_0$, where we write $\G{k}_0(x,y) = c_{n,k}\abs{x-y}^{2k-n}$ the Green's function for $\Dg[\xi]^k$ in $\R^n$.
    Now, we compute
    \begin{equation}\label{tmp:eqh}
        \begin{aligned}
            \Dg[\xi]^k R_{\a,x} &= (\Dg[\xi]+\a)^k \Ga[x]{k} - \Dg[\xi]^k \G{k}_{0,x} -\sum_{l=0}^{k-1}\tbinom{l}{k}\, \a^{k-l}\Dg[\xi]^l \Ga[x]{k}\\
                &= -\sum_{l=0}^{k-1}\tbinom{l}{k}\, \a^{k-l}\Dg[\xi]^l \Ga[x]{k}
        \end{aligned}
    \end{equation}
    in the distributional sense on $\R^n$. Let $h_\a(x,y) := -\sum_{l=0}^{k-1}\tbinom{l}{k}\, \a^{k-l}\Dg[\xi]^l \Ga[x]{k}(y)$, straightforward computations with \eqref{eq:Goest} then show that
    \begin{equation}\label{tmp:esth}
        \abs{h_\a(x,y)} \leq C \begin{cases}
            \a \abs{x-y}^{-(n-2)} & \sqa \abs{x-y} \leq 1\\
            \a^{k\frac{n+1}{4}} \abs{x-y}^{\frac{n(k-2)+k}{2}} e^{-\sqa \abs{x-y}} & \sqa \abs{x-y} \geq 1.
        \end{cases}.
    \end{equation}
    We now claim that, for all $y\neq x$,
    \begin{equation}\label{eq:repforR}
        R_\a(x,y) = \int_{\R^n} h_\a(x,z) \,c_{n,k}\abs{y-z}^{2k-n} dz.
    \end{equation}
    This follows from the fact that first, 
    \[  \abs{R_\a(x,y)} \leq \abs{\Ga{k}(x,y)} + \abs{\G{k}_0(x,y)} \leq C \abs{x-y}^{2k-n} \quad \forall~x\neq y.
    \] 
    Moreover, the right-hand side of \eqref{eq:repforR} defines a function in $L^1_{loc}(\R^n)$,
    \[  Z(x,y) := \int_{\R^n} h_{\a,x}(z)\, c_{n,k} \abs{y-z}^{2k-n} dz,
        \]
    which satisfies $\Dg[\xi]^k Z_x = h_{\a,x}$ in the distributional sense on $\R^n$. By Remark \ref{rem:uniqLap}, we conclude that for all $x\neq y$ in $\R^n$, $R_\a(x,y) = Z(x,y)$.
    We now have, by Lemma \ref{prop:expGir} together with \eqref{tmp:esth}, when $\sqa \abs{x-y} \leq 1$,
    \[  \abs{R_\a(x,y)} \leq \begin{bigcases}
        &C \a \abs{x-y}^{-(n-2k - 2)} & & \text{when $2k +2 < n$}\\
        &C \a \big(1+ \abs{\log \sqa \abs{x-y}}\big) & &\text{when $2k + 2 = n$}\\
        &C \a^{\frac{1}{2}} & &\text{when $2k + 1 = n$}
    \end{bigcases},
        \]
    where the constant $C>0$ does not depend on $\a > 0$. Finally, coming back to $\Ga{k} = \G{k}_0 + R_\a$, we have the conclusion.
\end{proof}

We obtain similar pointwise estimates for the derivatives of $\Ga{k}$. The next Corollary shows that the estimates in Proposition \ref{prop:excbvorGa} can be differentiated. 
\begin{corollary}\label{prop:betbvorDGak}
    Fix $k \geq 1$, $n>2k$, $\a >0$, and let $\Ga{k}$ the unique Green's function in $\clasG_k$ for the operator $(\Dg[\xi]+\a)^k$ in $\R^n$. For $l=1, \ldots 2k-1$, there exists $C_l > 0$ independent of $\a$ such that for all $x\neq y$ with $\sqa \abs{x-y} \leq 1$,
    \[
        \abs{\nabla^l \bpr{\abs{x-y}^{n-2k}\, \Ga[x]{k}(y)}} \leq C_l \abs{x-y}^{-l} \eta(\sqa\abs{x-y}),
        \]
    where $\eta$ is defined in \eqref{def:eta}. 
\end{corollary}
\begin{proof}
    With notations from the previous proof of Proposition \ref{prop:excbvorGa}, we have
    \[  \abs{x-y}^{n-2k} \Ga[x]{k}(y) = c_{n,k} + \abs{x-y}^{n-2k} R_{\a,x}(y),
        \]
    and thus for $l\geq 1$,
    \begin{equation}\label{tmp:derR}
        \abs{\nabla^l \bpr{\abs{x-y}^{n-2k}\, \Ga[x]{k}(y)}} = \abs{\nabla^l \bpr{\abs{x-y}^{n-2k}\, R_{\a,x}(y)}}.
    \end{equation}
    To estimate the derivatives of $R_\a$, we go back to \eqref{eq:repforR}: Since $l\leq 2k-1$, we can differentiate under the integral sign, and we obtain
    \[  \nabla^l R_{\a,x}(y) = \int_{\R^n} h_\a(x,z) \nabla^l_{(y)} \G{k}_0(y,z)\, dz.
        \]
    By Proposition \ref{prop:bvorDGak}, and since $l<2k$, we can use as before Lemma \ref{prop:expGir} to estimate the derivatives of $R_{\a,x}$ when $\abs{x-y} \leq 1/\sqa $ :
    \begin{itemize}
        \item When $n-2k = 1$, we have to consider several cases. For $l = 1$ we get
        \[  \abs{\nabla_{(y)} R_{\a,x}(y)} \leq C \a \big(1+\abs{\log\sqa\abs{x-y}}\big),
        \]
        and for $2\leq l \leq 2k-1$,
        \[  \abs{\nabla^l_{(y)} R_{\a,x}(y)} \leq C \a \abs{x-y}^{1-l}.
        \]
        \item When $n-2k \geq 2$, we obtain for all $1\leq l\leq 2k-1$,
        \[  \abs{\nabla^l_{(y)} R_{\a,x}(y)} \leq C \a \abs{x-y}^{-(n-2k-2+l)}.
        \]
    \end{itemize}
    By \eqref{tmp:derR} and by Leibniz's formula, we now have, for $l=1,\ldots 2k-1$
    \begin{multline*}
        \abs{\nabla^l \bpr{\abs{x-y}^{n-2k}\, \Ga[x]{k}(y)}} \leq C\sum_{m=0}^l \abs{x-y}^{n-2k-m} \abs{\nabla^{l-m}R_{\a,x}(y)}\\
        \leq \begin{bigcases}
            &C\sqa \abs{x-y}^{1-l} & &\text{when $n-2k =1$}\\
            &C\a \abs{x-y}^{2-l} \big(1+ \abs{\log\sqa\abs{x-y}}\big) & &\text{when $n-2k=2$}\\
            &C \a \abs{x-y}^{2-l} & &\text{when $n-2k\geq 3$}
        \end{bigcases}.
    \end{multline*}
\end{proof}

\section{Extending the construction to a Riemannian manifold}\label{sec:riem}
In this Section we construct the Green's function for $(\Dg+\a)^k$ on a manifold $M$. We follow the construction from Robert \cite{Rob10}.
We prove uniqueness, positivity, as well as estimates that explicitly depend on $\a$.
In the following, we will always consider $(M,g)$ to be a compact Riemannian manifold without boundary, of dimension $n>2k$ and with injectivity radius $\inj>0$. We let $\Dg := -\div_g(\nabla\cdot)$ be the Laplace-Beltrami operator in $M$. The notation $\Bal{x}{R}$ will represent a ball of radius $R>0$ and center $x$ either in $\R^n$ or on the manifold, depending on the context.
\par Theorem \ref{prop:bvorGga} is proved in several steps. 
We first define an approximate fundamental solution for $(\Dg+\a)^k$ in $M$ which is modelled on the Euclidean fundamental solution of $(\Dg[\xi] + \a)^k$. It satisfies the equation $(\Dg+\a)^k G_x = \delta_x$ up to error terms. We then iteratively improve the precision of these terms until we obtain a bounded error, which is finally controlled in subsection \ref{sec:remterm}.
Subsequently, we prove bounds on the derivatives of the Green's function of the same kind as Proposition \ref{prop:bvorDGak} and Corollary \ref{prop:betbvorDGak}. We finish this section with a remark on the mass of the operator $(\Dg+\a)^k$ when the dimension of the manifold $n=2k+1$.
\par We start with an observation.
\begin{lemma}\label{prop:uniqpos}
    Let $\a > 0$, if $G$ is a Green's function for the operator $(\Dg+\a)^k$ in $M$, as defined in Definition \ref{def:Green}, then it is unique.
\end{lemma}
\begin{proof}
    Start by noting that, by the same arguments as in Remark \ref{rem:smthGr}, any Green's function for $(\Dg+\a)^k$ is smooth away from its singularity, so that $G_x \in C^\infty(M\setminus\{x\})$ for any $x\in M$.
    \par Let $\Tilde{G}$ be another Green's function for $(\Dg+\a)^k$ in $M$. Take $\phi \in C^\infty(M)$, and define $u \in C^\infty$ the unique solution to the equation
    \[  (\Dg+ \a)^k u = \phi \qquad \text{on $M$.}
        \]
    We then have, for all $x\in M$,
    \[  \begin{bigcases}
        u(x) &= \intM{G(x,y) \phi(y)}(y)\\
        u(x) &= \intM{\Tilde{G}(x,y) \phi(y)}(y)
    \end{bigcases}.
        \]
    For all $x\in M$, $\phi\in C^\infty(M)$, we have obtained
    \[  \intM{\bpr{G(x,y) - \Tilde{G}(x,y)}\phi(x)} = u(x) - u(x) = 0.
        \]
    Thus, we can conclude that $G_x(y) = \Tilde{G}_x(y)$ for almost every $y\in M$, and by continuity $G_x(y)= \Tilde{G}_x(y)$ on $M\setminus\{x\}$.
\end{proof}

\subsection{Step 1: An approximate fundamental solution.}
We start the proof of Theorem \ref{prop:bvorGga} by pulling back the function $\Ga{k}$ onto the manifold. Fix $k\geq 1$, $n>2k$, and let $0< \tau_0 < \inj/2$ that will be chosen later. Let $\chi \in C^\infty(\R)$ be a cut-off function with $0 \leq \chi \leq 1$, such that $\chi(t) \equiv 1$ on $[0,\tau_0/2)$ and $\chi(t)\equiv 0$ on $(\tau_0, +\infty)$. We define
\begin{equation}\label{def:Gg}
    \Gg(x,y) := \chi(\dg{x,y}) \Ga{k}(0, \exp^{-1}_x(y))
\end{equation}
for all $x \neq y$ in $M$, where $\Ga{k}$ is defined in Theorem \ref{prop:uniqGa}. This function only depends on $\dg{x,y}$.
\par Assume that $\a$ is large enough so that $1/\sqa < \tau_0/2$. Using Theorem \ref{prop:uniqGa} we obtain
\begin{equation}\label{eq:estGg}
    \Gg(x,y) \leq \begin{cases}
    C \dg{x,y}^{-(n-2k)} & \sqa \dg{x,y} \leq 1\\
    C\a^{k \frac{n-3}{4}}\dg{x,y}^{\frac{n(k-2)+k}{2}} e^{-\sqa \dg{x,y}} & \sqa \dg{x,y} \geq 1
\end{cases}.
    \end{equation}
In particular, when $\sqa \dg{x,y} \to 0$, we have that
\begin{equation}\label{eq:decompGg}
    \Gg(x,y) = c_{n,k}\dg{x,y}^{2k-n}\bpr{1+\bigO\big(\eta(\sqa\dg{x,y})\big)},
\end{equation}
this function behaves to first order as a Riemannian version of the Green's function for the poly-Laplacian $\Dg[\xi]^k$ in $\R^n$, where $\eta$ is defined in \eqref{def:eta}.
\par The following Proposition estimates the error term between $\Gg$ and a true fundamental solution in the distributional sense. 
\begin{proposition}\label{prop:deflx}
    Let $\tau_0 < \inj/2$ and let $\Gg$ be as defined in \eqref{def:Gg}. There exist $\a_0\geq 1$ and $C>0$ such that, for all $\a \geq \a_0$ and $x\in M$, there is $l_{\a,x} \in C^0(M\setminus \{x\})$ satisfying
    \begin{equation}\label{eq:Galx}
        \intM{(\Dg+\a)^k \phi \, \Gg[x]} = \phi(x) + \intM{\phi(y) \, l_{\a,x}(y)}(y),
        \end{equation}
    for all $\phi \in C^{\infty}(M)$. The function $l_{\a,x}$ is $L^1(M)$, has support in $\Bal{x}{\tau_0}$, and
    \[  \abs{l_{\a,x}(y)} \leq \begin{cases}
        C \dg{x,y}^{-(n-2)} & \sqa \dg{x,y} \leq 1\\
        C \a^{k\frac{n+1}{2}}\dg{x,y}^{\frac{n(k-2)+k+4}{2}}e^{-\sqa \dg{x,y}} & \sqa \dg{x,y} \geq 1
    \end{cases}
        \]
    for all $x\neq y$ in $M$.
\end{proposition}
\begin{proof}
    We compute $(\Dg+\a)^k \Gg[x](y)$ and precisely estimate the error terms. Let $\Tilde{g} := \exp_x^* g$, it is a metric on $\Bal{0}{\tau_0} \subset \R^n$ with bounded geometry since $\tau_0 < \inj$. In particular, for $f \in C^2(\R^n)$,
    \begin{equation}\label{eq:errDgD}
        \Dg[\Tilde{g}]f(u) = \Dg[\xi]f(u) + \bigO(\abs{u}\abs{\nabla f(u)}) + \bigO(\abs{u}^2\abs{\nabla^2 f(u)}),
    \end{equation}
    where $u := \exp_x^{-1}(y) \in \Bal{0}{\tau_0}$. Since $\Gg(x,y) = \chi(\dg{x,y}) \Ga{k}(0,\exp^{-1}_x(y))$ is supported in $\Bal{x}{\tau_0}$, we can write, when $x\neq y$ with $\dg{x,y} \leq \tau_0$,
    \begin{equation*}
        \Dg \Gg[x](y) 
            = \Dg[\Tilde{g}]\bpr{\Ga[0]{k}(u)\chi(\abs{u})}\at{u = \exp_x^{-1}(y)}.
        \end{equation*}
        We observe that $\bpr{\Ga[0]{k}\,\chi(\abs{\cdot})} \in C^\infty(\R^n\setminus\{0\})$ is supported in $\Bal{0}{\tau_0}$. Now for $u \neq 0$ such that $\abs{u} < \tau_0/2$, using Proposition \ref{prop:bvorDGak}, we have
    \begin{align*}  
        \abs{\nabla^l \bpr{\chi(\abs{u})\Ga[0]{k}(u)}} &= \abs{\nabla^l \Ga[0]{k}(u)} \\
        &\leq C \begin{cases}
            \abs{u}^{-(n-2k+l)} & \sqa \abs{u} \leq 1\\
            \a^{k\frac{n-3}{4}+\frac{l}{2}} \abs{u}^{\frac{n(k-2)+k}{2}}e^{-\sqa \abs{u}} & \sqa \abs{u} \geq 1
        \end{cases}
        \end{align*}
    for $l=0, \ldots 2k$. Take $\a_0$ such that $1/\sqrt{\a_0} <\tau_0/2$, then we have $\sqa \abs{u} \geq 1$ for all $\abs{u}\geq \tau_0/2$, so that
    \begin{align*}
        \abs{\nabla^l \bpr{\chi(\abs{u})\Ga[0]{k}(u)}} &\leq \sum_{m=0}^{l}\abs{\nabla^{l-m}\chi}\abs{\nabla^m \Ga[0]{k}} \\
            &\leq C \a^{k\frac{n-3}{4} + \frac{l}{2}} \abs{u}^{\frac{n(k-2)+k}{2}} e^{-\sqa \abs{u}}.
    \end{align*}
    We have thus 
    \begin{equation}\label{eq:estGg0}
        \abs{\nabla^l \bpr{\chi(\abs{u}) \Ga[0]{k}(u)}} \leq C\begin{cases}
            \abs{u}^{-(n-2k+l)} & \sqa \abs{u} \leq 1\\
            \a^{k\frac{n-3}{4} + \frac{l}{2}} \abs{u}^{\frac{n(k-2)+k}{2}} e^{-\sqa \abs{u}} & \sqa \abs{u} \geq 1\\
            0 & \abs{u} > \tau_0
        \end{cases}.
    \end{equation}
    \par We now show that, for $\phi \in C^{\infty}(M)$, 
    \begin{equation}\label{tmp:explx}
        \intM{(\Dg+\a)^k \phi\, \Gg[x]} = \phi(x) + \lim_{\delta \to 0} \intM[M\setminus\Bal{x}{\delta}]{\phi\,(\Dg+ \a)^k \Gg[x]}.
    \end{equation}
    Start by observing that since $\Gg[x] \in L^1(M)$,
    \begin{align*}  
        \intM{(\Dg+\a)^k \phi\, \Gg[x]} &= \lim_{\delta\to 0} \intM[M\setminus\Bal{x}{\delta}]{(\Dg + \a)^k \phi\, \Gg[x]}\\
            &= \lim_{\delta\to 0}\bsq{\sum_{l=0}^{k} \tbinom{l}{k} \a^{k-l} \intM[M\setminus\Bal{x}{\delta}]{\Dg^l \phi \, \Gg[x]}}.
    \end{align*}
    Integration by parts gives, as in \eqref{eq:decGabc}, for $l=1,\ldots k$,
    \begin{multline}\label{tmp:ebypart}
        \intM[M\setminus \Bal{x}{\delta}]{\Dg^l \phi\, \Gg[x]} = \intM[M\setminus\Bal{x}{\delta}]{\phi\, \Dg^l \Gg[x]}\\ + \sum_{m=0}^{l-1}\int_{\partial\Bal{x}{\delta}} \partial_\nu \Dg^{l-1-m}\phi \Dg^m \Gg[x] d\sigma_g\\ - \sum_{m=0}^{l-1}\int_{\partial\Bal{x}{\delta}} \Dg^{l-1-m}\phi\, \partial_\nu \Dg^m \Gg[x] d\sigma_g,
    \end{multline}
    where $\partial_\nu$ is the covariant derivative along the normal direction to $\partial\Bal{x}{\delta}$ in $M$. Using \eqref{eq:estGg0}, we obtain for $l=1,\ldots k$, $m = 0,\ldots l-1$,
    \begin{equation}\label{tmp:abord1}
        \abs{\int_{\partial\Bal{x}{\delta}} \partial_\nu \Dg^{l-1-m}\phi \Dg^m \Gg[x] d\sigma_g} \leq C \delta^{-n+2k-2m} \int_{\partial\Bal{x}{\delta}}d\sigma_g = o(1)
    \end{equation}
    as $\delta \to 0$. On the other hand, when $l=k$ and $m=k-1$, we compute
    \begin{multline}\label{tmp:bbycont}
        -\int_{\partial\Bal{x}{\delta}} \phi(y) \partial_\nu \Dg^{k-1}\Gg[x](y) d\sigma_g(y)\\ = \phi(x) \int_{\partial\Bal{x}{\delta}} -\partial_\nu \Dg^{k-1}\Gg[x](y)d\sigma_g(y) + o(1) \quad \text{as $\delta\to 0$}
    \end{multline}
    since $\phi$ is $C^\infty(M)$. Using \eqref{eq:errDgD} iteratively, we have
    \begin{multline*}
        \int_{\partial\Bal{x}{\delta}} -\partial_\nu \Dg^{k-1}\Gg[x](y) d\sigma_g(y)\\
        \begin{aligned}  
            &= \int_{\partial\Bal{0}{\delta}} -\partial_\nu\bpr{\Dg[\xi]^{k-1}\Ga[0]{k}(u)}d\sigma + \bigO\bpr{\int_{\partial\Bal{0}{\delta}} \delta^2 \abs{\nabla^{2k-1}\Ga[0]{k}(u)} d\sigma}\\
                &\qquad + \bigO\bpr{\int_{\partial\Bal{0}{\delta}}\delta \abs{\nabla^{2k-2}\Ga[0]{k}(u)}d\sigma} + \bigO\bpr{\sum_{m=2}^{2k-3}\int_{\partial\Bal{0}{\delta}}\abs{\nabla^m \Ga[0]{k}(u)}d\sigma}\\
                &= \int_{\partial\Bal{0}{\delta}} -\partial_\nu \bpr{\Dg[\xi]^{k-1} \Ga[0]{k}(u)} d\sigma + o(1), 
        \end{aligned}  
    \end{multline*}
    as $\delta\to 0$, estimating the terms with \eqref{eq:estGg0}. Here in the right-hand side, $\partial_\nu$ is now the derivative normal to the sphere in the Euclidean space. Using Corollary \ref{prop:betbvorDGak}, one has 
    \[  \int_{\partial\Bal{0}{\delta}} -\partial_\nu \bpr{\Dg[\xi]^{k-1} \Ga[0]{k}(u)} d\sigma = \int_{\partial\Bal{0}{\delta}} -\partial_\nu \bpr{\Dg[\xi]^{k-1} \frac{c_{n,k}}{\abs{u}^{n-2k}}} d\sigma + o(1),
        \]
    and thus by definition of $c_{n,k}$ in \eqref{def:cnk}, we obtain in the end
    \begin{equation}\label{tmp:cest1}
        \int_{\partial\Bal{x}{\delta}} -\partial_\nu \Dg^{k-1}\Gg[x] d\sigma_g = 1 + o(1) \qquad \text{when $\delta\to 0$}.
    \end{equation}
    Finally, all the remaining terms are estimated with \eqref{eq:estGg0},
    \begin{equation}\label{tmp:dbord2}
        \abs{\int_{\partial\Bal{0}{\delta}} \Dg^{l-1-m}\phi\, \partial_\nu \Dg^{m}\Gg[x] d\sigma_g} \leq C\delta^{2k-m-1}= o(1)
    \end{equation}
    as $\delta \to 0$, when $m\neq k-1$, $1\leq l\leq k$. Coming back to \eqref{tmp:ebypart}, putting together \eqref{tmp:bbycont}, \eqref{tmp:cest1} in \eqref{tmp:abord1}, and with \eqref{tmp:dbord2}, we obtain \eqref{tmp:explx}.
    Note that all the terms $\Dg^l \Gg[x]$ are integrable by \eqref{eq:estGg0} except for the term $\Dg^k \Gg[x]$ which is only bounded by $\dg{x,y}^{-n}$ when $\dg{x,y} \leq 1/\sqrt{\a}$.
    \par Write $l_{\a,x}(y) := (\Dg + \a)^k \Gg[x](y)$ for all $x\neq y$, then $l_{\a,x} \in C^0(M \setminus \{x\})$ with support in $\Bal{x}{\tau_0}$ by the definition of $\Gg$. Now for $x\neq y$ such that $\dg{x,y} \leq \tau_0 $, write $u := \exp_x^{-1}(y)$, and compute
    \begin{multline*}
        (\Dg + \a)^k \Gg[x](\exp_x(u)) \begin{aligned}[t]
            &= (\Dg[\Tilde{g}]+\a)^k \bpr{\chi(\abs{u})\Ga[0]{k}(u)}\\
        \end{aligned}\\
            = \sum_{l=0}^k \tbinom{l}{k} \a^{k-l}\Bigg[\chi(\abs{u}) \Dg[\xi]^l \Ga[0]{k}(u) + \bigO\bpr{\sum_{m=1}^{2l-2}\abs{\nabla^{m}\bpr{\chi(\abs{u})\Ga[0]{k}(u)}}}\\
            + \bigO\bpr{\abs{u}\abs{\nabla^{2l-1}\bpr{\chi(\abs{u})\Ga[0]{k}(u)}}} + \bigO\bpr{\abs{u}^2 \abs{\nabla^{2l}\bpr{\chi(\abs{u})\Ga[0]{k}(u)}}}\Bigg],
    \end{multline*}
    using again \eqref{eq:errDgD} iteratively. By \eqref{eq:estGg0} and since $\Ga{k}$ is a fundamental solution for the operator $(\Dg[\xi]+\a)^k$ in $\R^n$, we have when $0< \abs{u} \leq 1/\sqa$,
    \begin{align*}
        \abs{(\Dg+\a)^k \Gg[x](\exp_x(u))} &= \sum_{l=0}^k \a^{k-l} \Bigg[\bigO\bpr{\abs{u}^2 \abs{\nabla^{2l}\Ga[0]{k}(u)}}\\
        &\quad  + \bigO\bpr{\abs{u}\abs{\nabla^{2l-1}\Ga[0]{k}(u)}} + \bigO\bpr{\sum_{m=1}^{2l-2}\abs{\nabla^{m}\Ga[0]{k}(u)}} \Bigg]\\
        &\leq C \sum_{l=0}^k \a^{k-l} \abs{u}^{2k-n-2l+2} \leq \abs{u}^{2-n}.
    \end{align*}
    Similarly, when $\sqa\abs{u} \geq 1$,
    \[  \abs{(\Dg+\a)^k \Gg[x](\exp_x(u))} \leq C\sum_{l=0}^k \a^{k-l} \abs{u}^2 \a^{k\frac{n-3}{4}+l}\abs{u}^{\frac{n(k-2)+k}{2}} e^{-\sqa \abs{u}}.
        \]
    We have thus shown that, for $x\neq y$,
    \begin{equation}  \label{eq:estlx}
        \abs{l_{\a,x}(y)} \leq C \begin{cases}
        \dg{x,y}^{-(n-2)} & \sqa \dg{x,y} \leq 1\\
        \a^{k\frac{n+1}{4}}\dg{x,y}^{\frac{(k-2)n+k+4}{2}} \expab{x,y} & \sqa \dg{x,y} \geq 1\\
        0 & \dg{x,y} \geq \tau_0
    \end{cases}.
    \end{equation}
    In particular, $l_{\a,x} \in L^1(M)$ so that \eqref{tmp:explx} becomes
    \[  \intM{(\Dg+\a)^k \phi\, \Gg[x]} = \phi(x) + \intM{\phi\,l_{\a,x}}
        \]
    for all $\phi \in C^{\infty}(M)$. Note that $k\frac{n+1}{4} - \frac{(k-2)n+k+4}{2} = \frac{n-2}{2}$, so that the two regimes in \eqref{eq:estlx} are of order $\a^{\frac{n-2}{2}}$ when $\dg{x,y} \sim 1/\sqa $.
\end{proof}

\subsection{Step 2: The induction step.}
In this step, we define a sequence of functions to iteratively improve the estimates on the error term.
\begin{proposition}\label{prop:defgamx}
    There exists $N\in \N$, $\tau_0>0$, $\a_0 \geq 1$ such that, for all $\a\geq \a_0$, there is a function $\Gg^* \in C^\infty(M\times M \setminus \Diag)$ and for all $x\in M$, a function $\gamma_{\a,x} \in C^0(M)$ such that
    \begin{equation}\label{eq:compG*}
        \intM{(\Dg+\a)^k \phi(y) \Gg^*(x,y)}(y) + \intM{\gamma_{\a,x}(y) \phi(y)}(y) = \phi(x)
    \end{equation}
    for all $\phi \in C^{\infty}(M)$. Moreover, both $\Gg[x]^*$ and $\gamma_{\a,x}$ are supported in $\Bal{x}{N\tau_0}$, $\Gg^*$ satisfies 
    \begin{equation}\label{eq:estG*}
        \Gg^*(x,y) = c_{n,k} \dg{x,y}^{2k-n}\bpr{1+\bigO\bpr{\eta(\sqa \dg{x,y})}}
    \end{equation}
    for $\sqa \dg{x,y} \leq 1$, and there is a constant $C>0$ independent of $\a \geq \a_0$ such that writing $p_{n,k} := \frac{k(n+1)+4}{2}N-n$, 
    \begin{equation}\label{eq:estgamx}
        \abs{\gamma_{\a,x}(y)} \leq C \a^{-N+\frac{n}{2}} \begin{cases}
        1 & \sqa \dg{x,y} \leq 1\\
        \bpr{\sqa \dg{x,y}}^{p_{n,k}} \expab{x,y} & \sqa \dg{x,y} \geq 1\\
        0 & \dg{x,y} \geq N\tau_0
    \end{cases},
        \end{equation}
    for all $x,y \in M$.
\end{proposition}
\begin{proof}
Fix $\tau_0>0$ such that $\tau_0 < \inj/(n+2)$, and take $\a_0\geq 1$ given by Proposition \ref{prop:deflx}. We define, for all $\a \geq \a_0$ and all $x\neq y$ in $M$,
\begin{align*}
    \Gam{1}(x,y) &:= -l_{\a,x}(y)\\
    \Gam{i+1}(x,y) &:= \intM{\Gam{i}(x,z) \Gam{1}(z,y)}(z) \qquad \forall~i\geq 2.
\end{align*}
By the exponential version of Giraud's Lemma, Lemma \ref{prop:rexpGir}, and by \eqref{eq:estlx} we see that, as long as $2i < n$ and $i\tau_0 < \inj$, we have $\Gam[x]{i} \in L^1(M),\, \Gam{i} \in C^0(M\times M \setminus \Diag)$, and
\[  \abs{\Gam{i}(x,y)} \leq C_i \begin{cases}
    \dg{x,y}^{-(n-2i)} & \sqa \dg{x,y} \leq 1\\
    \a^{ki\frac{n+1}{4}} \dg{x,y}^{\frac{k(n+1) + 4}{2}i - n} \expab{x,y} & \sqa \dg{x,y} \geq 1\\
    0 & \dg{x,y} \geq i\tau_0
\end{cases}.
    \]
Now take $N = \left\lfloor \tfrac{n}{2}\right\rfloor +1 \in \N$, so that $2N > n$. By the choice of $\tau_0$, we have $N\tau_0 < \inj/2$. Lemma \ref{prop:rexpGir} then shows that $\Gam{N} \in C^0(M\times M)$ and that
\begin{multline}\label{eq:estGamN}
    \abs{\Gam{N}(x,y)}\\ \leq C \begin{cases}
    \a^{-N+\frac{n}{2}}   &\sqa\dg{x,y}\leq 1\\
    \a^{kN\frac{n+1}{4}} \dg{x,y}^{\frac{k(n+1) + 4}{2}N - n} \expab{x,y}   &\sqa\dg{x,y}\geq 1\\
    0  &\dg{x,y}\geq N\tau_0
\end{cases}.
    \end{multline}
Let, for $i=1, \ldots N-1$ and $x\neq y$,
\begin{equation}\label{eq:defggi}  
    \Gg^i(x,y) := \intM{\Gam{i}(x,z) \Gg(z,y)}(z).
\end{equation}
If $x\neq y$ are such that $\sqa \dg{x,y} \leq 1$, we again have by Lemma \ref{prop:rexpGir}
\begin{equation}\label{eq:estTGsqa}
    \abs{\Gg^i(x,y)} \leq C_i \begin{cases}
        \dg{x,y}^{-(n-2k -2i)} & \text{when $2k+2i < n$}\\
        1 + \abs{\log(\sqa \dg{x,y})} & \text{if $2k+2i = n$}\\
        \a^{-\frac{2k+2i-n}{2}} & \text{when $2k + 2i > n$}
    \end{cases}.
\end{equation}
While, if $\sqa \dg{x,y} \geq 1$, writing for simplicity
\begin{align}\label{eq:defpiri}
    p_i &:= k\tfrac{i(n+1)+n-3}{4} & &\text{and} & \rho_i &:= \tfrac{k(n+1)(i+1)-2n+4i}{2}
\end{align}
satisfying $2p_i - \rho_i = n-2k - 2i$, we have
\begin{equation}\label{eq:estTG}
    \abs{\Gg^i(x,y)} \leq C_i \a^{p_i} \dg{x,y}^{\rho_i}\expab{x,y}
\end{equation}
with $\Gg^i(x,y) = 0$ when $\dg{x,y} \geq (i+1)\tau_0$.
\par In some sense, the $\Gg^i$ are successive error terms in the expression of the Green's function. Let us define for all $x\neq y$ in $M$,
\begin{equation}\label{def:Gg*}
    \Gg^*(x,y) := \Gg(x,y) + \sum_{i=1}^{N-1} \Gg^i(x,y).
    \end{equation}
Then we have, for all $\phi \in C^{\infty}(M)$ and $x \in M$, by Proposition \ref{prop:deflx} and since $\Gam[x]{1} = -l_{\a,x}$,
\begin{multline*}
    \intM{(\Dg + \a)^k \phi(y) \Gg^*(x,y)}(y) \\ 
    \begin{aligned}
        &= \intM{(\Dg+\a)^k \phi\, \Gg[x]}\\
        &\qquad\qquad + \sum_{i=1}^{N-1} \intM{\intM{\Gam{i}(x,z)\Gg(z,y)(\Dg+\a)^k\phi(y)}(z)}(y)\\
        &= ~\phi(x) + \intM{l_{\a,x}(y)\phi(y)}(y)\\
        &\qquad\qquad + \sum_{i=1}^{N-1} \intM{\Gam{i}(x,z)\bpr{\intM{\Gg[z](\Dg+\a)^k\phi}}}(z)\\
        &= ~\phi(x) + \intM{l_{\a,x}\,\phi} + \sum_{i=1}^{N-1} \intM{\Gam{i}(x,z) \bsq{\phi(z) + \intM{l_{\a,z}\, \phi}}}(z)
    \end{aligned}
\end{multline*}
where we used Fubini. Therefore, we obtain with the definition of $\Gam{i+1}$ that 
\begin{align*}
    \intM{(\Dg + \a)^k \phi(y) \Gg^*(x,y)}(y) &= ~\phi(x) - \intM{\Gam[x]{1} \phi} + \sum_{i=1}^{N-1} \intM{\Gam[x]{i} \phi}\\
        &\hfil+ \sum_{i=1}^{N-1} \intM{\bsq{\intM{\Gam{i}(x,z)l_z(y)}(z)} \phi(y)}(y)\\
        &= ~\phi(x) - \intM{\Gam[x]{N} \phi}
\end{align*} 
\par We now let $\gamma_{\a,x}(y) := \Gam{N}(x,y)$, then \eqref{eq:compG*} follows, and $\gamma_{\a,x} \in C^0(M)$ satisfies \eqref{eq:estgamx} thanks to \eqref{eq:estGamN}. Finally, for $\sqa\dg{x,y} \leq 1$, and again by Lemma \ref{prop:rexpGir}, we have
\begin{equation}\label{eq:estsumGi}
    \begin{aligned}
        \abs{\sum_{i=1}^{N-1} \Gg^i(x,y)} &\leq \begin{bigcases}
            & C\a^{-1/2} & \text{if $n-2k= 1$,}&\\
            & C\big(1+\abs{\log\sqa\dg{x,y}}\big) & \text{if $n-2k= 2$,}&\\
            & C\dg{x,y}^{-(n-2k-2)} & \text{if $n-2k\geq 3$,}&
        \end{bigcases}\\
            & = C \a^{-1}\dg{x,y}^{-(n-2k)}\eta(\sqa\dg{x,y}),
    \end{aligned}
\end{equation}
recalling the definition \eqref{def:eta} for $\eta$. This means that, with \eqref{eq:decompGg}, we still have
\[
    \Gg^*(x,y) = c_{n,k}\dg{x,y}^{-(n-2k)} \bpr{1 +\bigO\big(\eta(\sqa\dg{x,y})\big)}
\]
when $\sqa \dg{x,y} \leq 1$.
\end{proof}

\subsection{Step 3: Estimates on the remainder term.}\label{sec:remterm} 
With Proposition \ref{prop:defgamx}, we have modified the starting function $\Gg$ to get closer to a real fundamental solution to the operator $(\Dg+\a)^k$ in $M$. The remainder $\gamma_{\a,x}$ is now uniformly bounded in $\a$ by \eqref{eq:estgamx} and continuous in $M\times M$.
\par Fix $\a \geq \a_0$ given by the Proposition \ref{prop:defgamx}, and $x \in M$. We let $u_{\a,x}$ be the unique solution of
\begin{equation}\label{eq:equx}
    (\Dg + \a)^k u_{\a,x} = \gamma_{\a,x} \qquad \text{weakly in $M$,}
\end{equation}
where $\gamma_{\a,x} = \Gam[x]{N}$ as in the proof of Proposition \ref{prop:defgamx}.
Such $u_{\a,x}$ exists and is unique since $(\Dg+\a)^k$ is coercive for any $\a>0$, and since $\gamma_{\a,x} \in C^0(M)$.
\begin{remark}\label{rem:adiverg}
    By standard elliptic theory and Sobolev's embeddings, we have $u_{\a,x} \in C^{2k-1, \theta}(M)$ for all $\theta \in (0,1)$, with
    \begin{equation}\label{eq:1stestux}
        \norm{u_{\a,x}}_{C^l(M)} \leq C_{l,\a} \norm{\gamma_{\a,x}}_{C^0(M)} \leq C_{l,\a} \qquad l=0,\dots 2k-1
    \end{equation}
    by \eqref{eq:estgamx}.
    The constants $C_{l,\a}$ depend on $\a$, but we want pointwise estimates on $u_{\a,x}$ and its derivatives with an explicit dependence in $\a$, so that \eqref{eq:1stestux} is not enough, unlike in \cite{Rob10} where \eqref{eq:1stestux} leads to the conclusion. In particular, in this step, we aim at recovering some exponential decay for $u_{\a,x}(y)$, when $\sqa \dg{x,y}\geq 1$.
\end{remark}

\begin{proposition}\label{prop:defGga}
    Let $\a \geq \a_0$ given by Proposition \ref{prop:defgamx}. For $\a\geq \a_0$ and $x\in M$, let $\Gg^*$ be as in \eqref{def:Gg*}, and let $u_{\a,x} \in \Sob(M)$ be the unique weak solution to \eqref{eq:equx}. Define
    \begin{equation}\label{def:Gga}
        \Gga(x,y) := \Gg^*(x,y) + u_{\a,x}(y)
    \end{equation}
    for all $x\neq y$ in $M$. Then $\Gga$ is a fundamental solution for the operator $(\Dg+\a)^k$ in $M$. Moreover, if $f \in L^q(M)$ for some $q>n/2k$, and $\Tilde{u} \in \Sob[k,q](M)$ solves $(\Dg+\a)^k \Tilde{u} = f$ in $M$, then $\Tilde{u} \in C^{0}(M)$ and we have the following representation formula,
    \[  \Tilde{u}(x) = \intM{\Gga(x,z)f(z)}(z) \quad \text{for all $x\in M$}.
        \]
\end{proposition}
\begin{proof}
    For the first part, go back to \eqref{eq:compG*}. For all $\phi \in C^\infty(M)$,
    \begin{multline}\label{eq:GgaisG}
        \intM{(\Dg+\a)^k \phi(y)\,\Gga(x,y)}(y)\\
        \begin{aligned}
            &= \phi(x) - \intM{\gamma_{\a,x} \phi} + \intM{u_{\a,x} (\Dg+ \a)^k\phi}\\
            &= \phi(x)    
        \end{aligned}
    \end{multline}
    since $u_{\a,x}$ is a weak solution to \eqref{eq:equx}. By \eqref{eq:estG*}, and since $u_{\a,x}$ is continuous in $M$, we have $\Gga[x] \in L^1(M)$ and $\Gga$ is a fundamental solution for the operator $(\Dg+\a)^k$ in $M$. 
    \par Now take $f \in L^q(M)$, $q> n/2k$ and $\Tilde{u} \in \Sob(M)$ satisfying $(\Dg+\a)^k \Tilde{u} = f$ weakly in $M$. Standard elliptic theory gives $\Tilde{u} \in \Sob[2k,q](M)$, and then with Sobolev's embeddings we have $\Tilde{u}\in C^{0,\theta}(M)$ for some $\theta \in (0,1)$. Let $(f_m)_{m\geq 1}$ be a sequence of functions in $C^\infty(M)$ such that $f_m \to f$ in $L^q$, and take $\Tilde{u}_m \in C^{\infty}(M)$ the respective solutions to $(\Dg+\a)^k \Tilde{u}_m = f_m$. By elliptic estimates and Sobolev's embedding, for any fixed $\a\geq 1$, $(\Tilde{u}_m)_{m\geq 1}$ is a bounded sequence in $C^{0,\theta}(M)$. Then, by compactness of the inclusion $C^{0,\theta}(M) \sub C^0(M)$ and since the solution is unique, $\Tilde{u}_m \to \Tilde{u}$ in $C^{0}(M)$ up to a subsequence. Since $u_m \in C^\infty(M)$ we have, testing \eqref{eq:GgaisG} against $u_m$,
    \[  \Tilde{u}_m(x) = \intM{\Gga(x,y) f_m(y)}(y).
        \]
    And, for all $x\in M$, $\Tilde{u}_m(x) \to^{m\to \infty} \Tilde{u}(x)$. Finally, for all $1 \leq p < \frac{n}{n-2k}$, we have $\Gga[x] \in L^p(M)$ since $\Gg[x]^* \in L^p(M)$ by \eqref{eq:estG*} and $u_{\a,x} \in C^0(M) \sub L^p(M)$. Thus, since $f_m \to f$ in $L^q(M)$, choosing $1 \leq p < \frac{n}{n-2k}$ such that $\tfrac{1}{p}+\tfrac{1}{q} = 1$, Hölder's inequality implies that
    \begin{align*}  
        \abs{\intM{\Gga(x,y) f_m(y)} - \intM{\Gga(x,y)f(y)}} &\leq \Lnorm[(M)]{p}{\Gga[x]} \Lnorm[(M)]{q}{f_m -f}\\
            &\to 0.
    \end{align*}
    We have thus shown, for all $x\in M$,
    \[  \Tilde{u}(x) = \intM{\Gga(x,y)f(y)}(y).
        \]
\end{proof}

Thanks to Proposition \ref{prop:defGga} and \eqref{eq:equx}, there is a representation formula for $u_{\a,x}$ itself, given by
\begin{equation}\label{eq:repux}  
    \begin{aligned}
        u_{\a,x}(y) &= \intM{\Gga(y,z) \gamma_{\a,x}(z)}(z)\\
            &= \intM{\Gg^*(y,z) \gamma_{\a,x}(z)}(z) + \intM{u_y(z)\gamma_{\a,x}(z)}(z),
    \end{aligned}
    \end{equation}
where $\Gg^*$ and $\gamma_{\a,x}$ are introduced in Proposition \ref{prop:defgamx}. We now use this formula to self-improve the estimates on $u_{\a,x}$.
We prove exponential decay on $u_{\a,x}$, when $\a\geq \a_0$ is large enough. 
This is a striking difference with the case of operators with bounded coefficients.
\begin{proposition}\label{prop:bvorux}
    There exists $\a_0\geq 1$ such that the following holds. For all $0<\epsilon<1$, there is a constant $C_\epsilon>0$ such that for all $\a \geq \a_0$ and all $x,y \in M$,
    \begin{equation}\label{eq:estux}
        \abs{u_{\a,x}(y)} \leq C_\epsilon \a^{-k} \begin{cases}
        1 & \sqa \dg{x,y} \leq 1\\
        e^{-(1-\epsilon)\sqa \dg{x,y}} & \sqa \dg{x,y} \geq 1\\
        e^{-(1-\epsilon)\sqa \,\inj/2} & \dg{x,y} \geq \inj/2
    \end{cases},
        \end{equation}
    where $u_{\a,x}$ is the unique weak solution to \eqref{eq:equx}.
\end{proposition}

\begin{proof}
    We define, for all $\a\geq 1$ and $0<\epsilon<1$, a function $\Psi_{\epsilon,\a} \in L^1(M \times M)$ as
\begin{equation}\label{eq:defPea}
    \Pea(x,y) = \begin{cases}
    e^{-(1-\epsilon)} & \sqa \dg{x,y} \leq 1\\
    e^{-(1-\epsilon)\sqa \dg{x,y}} & \sqa \dg{x,y} \geq 1\\
    e^{-(1-\epsilon)\sqa \,\inj/2} & \dg{x,y} \geq \inj/2
\end{cases}.
    \end{equation}
    \par We prove the Proposition by using the representation formula \eqref{eq:repux} for $u_{\a,x}$. The first term is estimated with Lemma \ref{prop:rexpGir}, the estimates \eqref{eq:estTGsqa}, \eqref{eq:estTG} and \eqref{eq:estgamx} give
    \begin{multline}\label{eq:Ggam}
        \abs{\intM{\Gg^*(y,z)\gamma_{\a,x}(z)}(z)}\\ \leq C\begin{bigcases}
            &\a^{-N + \frac{n-2k}{2}} &  &\sqa \dg{x,y} \leq 1\\
            &\sum_{i=0}^N \a^{\Tilde{p}_i} \dg{x,y}^{\Tilde{\rho}_i} \expab{x,y} &  &\sqa \dg{x,y} \geq 1\\
            &0 &  &\dg{x,y} \geq 2N\tau_0
    \end{bigcases}
    \end{multline}
    where 
    \[  \Tilde{p}_i := k\tfrac{n+1}{4}(N+i) + k \tfrac{n-3}{4} \text{ and } \Tilde{\rho}_i := \tfrac{k(n+1)(N+i+1)}{2} + 2(N+i) -n, 
    \] 
    with $N\tau_0 < \inj/2$ and $N > n/2$, $N$ and $\tau_0$ are as defined in the proof of Proposition \ref{prop:defgamx}. 
    \par Fix any $0<\epsilon<1$ and $x\in M$, there is a constant $C_\epsilon>0$ independent of $\a$ and $x$ such that, for $\a$ large enough, we can write \eqref{eq:Ggam} as
    \begin{multline}\label{tmp:estepsGg*}
        \abs{\intM{\Gg^*(y,z)\gamma_{\a,x}(z)}(z)}\\ 
        \begin{aligned}
            &\leq C_\epsilon \a^{-N+\frac{n-2k}{2}}\begin{cases}
            1 & \sqa \dg{x,y}\leq 1\\
            e^{-(1-\epsilon)\sqa \dg{x,y}} & \sqa \dg{x,y}  \geq 1\\
            0 & \dg{x,y}  \geq 2\tau
            \end{cases}\\
            &\leq C_\epsilon \a^{-N+\frac{n-2k}{2}} \Pea(x,y),
        \end{aligned}
    \end{multline}
    for all $y \in M$.
    Note that, by integration by parts and the fact that $u_x$ solves \eqref{eq:equx}, one obtains 
    \begin{equation}\label{eq:uxsym}
    \begin{aligned}
        \intM{u_y(z)\gamma_{\a,x}(z)}(z) &= \intM{u_y(z)\, (\Dg+\a)^k u_x(z)}(z)\\
         &= \intM{(\Dg+\a)^k u_y(z)\, u_x(z)}(z)\\
         & = \intM{\gamma_{\a,y}(z)\, u_x(z)}(z).
    \end{aligned}
    \end{equation}
    \par We now claim that $u_{\a,x} = o(\Pea)$ in the sense that
    \[  \Lnorm[(M)]{\infty}{\frac{u_{\a,x}(\cdot)}{\Pea(x,\cdot)}} \to 0 \qquad \text{as $\a \to \infty$, when $x\in M$ is fixed}.
        \]
    For this, define for all $\a$,
    \[  
        \Upsilon_{\a,x} := \Lnorm[(M)]{\infty}{\frac{u_{\a,x}}{\Pea(x,\cdot)}} \\
    \]
    and let $y_\a \in M$ be such that $\frac{\abs{u_{\a,x}(y_\a)}}{\Pea(x,y_\a)} = \Upsilon_{\a,x}$.
    We know $\Upsilon_{\a,x}$ and $y_\a$ exist since $u_{\a,x}$, $\Pea(x,\cdot)$ are continuous, and $\Pea >0$ in $M$. Applying \eqref{eq:repux} at the point $y_\a$ and using \eqref{tmp:estepsGg*} and \eqref{eq:uxsym} we now have, for all $\a$ large enough,
    \begin{equation}\label{tmp:uxgametc}\begin{aligned}
        \abs{u_{\a,x}(y_\a)} &\leq \abs{\intM{\Gg^*(y_\a, z) \gamma_{\a,x}(z)}(z)} + \abs{\intM{u_{x}(z) \gamma_{\a,y_\a}(z)}(z)}\\
            &\leq C_\epsilon \a^{-N +\frac{n-2k}{2}} \Pea(x,y_\a) + \Upsilon_{\a,x} \intM{\Pea(x,z)\abs{\gamma_{\a,y_\a}(z)}}(z).
    \end{aligned}\end{equation}
    Using the bounds \eqref{eq:estgamx} on $\gamma_{\a,x}$ and Lemma \ref{prop:girPsi} below, which is a modified version of the exponential Giraud's Lemma, we obtain that
    \begin{equation}\label{eq:selfimprov}
        \abs{\intM{\Pea(x,z)\gamma_{\a,y_\a}(z)}(z)} \leq C_\epsilon' \a^{-N} \Pea(x,y_\a).
    \end{equation}
    Going back to \eqref{tmp:uxgametc} and dividing by $\Pea(x,y_\a)>0$, we have thus proven that
    \begin{equation}\label{eq:tmpups}
        \Upsilon_{\a,x} \leq C_\epsilon \a^{-N + \frac{n-2k}{2}} + C_\epsilon'\a^{-N} \Upsilon_{\a,x},
    \end{equation}
    with $N>n/2$ by definition of $N$. Hence, $\Upsilon_{\a,x} \to 0$ for all $x\in M$.
    \par We have thus shown, since $M$ is compact, that there exists $\a_0 \geq 0$ and a constant $C_\epsilon>0$ independent of $x$ such that for all $\a\geq \a_0$, $x,y \in M$,
    \begin{equation}\label{tmp:Peagir}  
        \abs{u_{\a,x}(y)} \leq C_\epsilon \begin{cases}
        1 & \sqa \dg{x,y} \leq 1\\
        e^{-(1-\epsilon)\sqa \dg{x,y}} & \sqa \dg{x,y} \geq 1\\
        e^{-(1-\epsilon)\sqa \inj/2} & \dg{x,y} \geq \inj/2.
    \end{cases}
        \end{equation}
    We now use this estimate to compute again the second term of \eqref{eq:repux}: Using Lemma \ref{prop:girPsi} with \eqref{tmp:Peagir}, and with \eqref{tmp:estepsGg*}, we finally get
    \[
        \abs{u_{\a,x}(y)} \leq C_\epsilon \a^{-k} \begin{cases}
            1 & \sqa \dg{x,y} \leq 1\\
            e^{-(1-\epsilon)\sqa \dg{x,y}} & \sqa \dg{x,y} \geq 1\\
            e^{-(1-\epsilon)\sqa \inj/2} & \dg{x,y} \geq \inj/2,
        \end{cases}
    \]
    since $-N +\frac{n-2k}{2} \leq -k$, which concludes the proof.
\end{proof}
\begin{remark}\label{rem:ageqa0}
    The assumption $\a \geq \a_0$, with $\a_0$ large, is crucial to obtain \eqref{tmp:Peagir} from \eqref{eq:tmpups}, and thus exponential decay for $u_{\a,x}$ when $\sqa\dg{x,y}\geq 1$.
\end{remark}
\begin{remark}
    With our approach, we cannot expect to obtain the exact decay $\expab{x,y}$ for $u_{\a,x}$ when $\sqa\dg{x,y} \geq 1$. Successive convolutions in the second term of the representation formula \eqref{eq:repux} add positive exponents of $\dg{x,y}$ that we cannot get rid of, see Remark \ref{rem:convGiraud}. We are thus forced to reduce the exponential decay of $u_{\a,x}$ to $e^{-(1-\epsilon)\sqa \dg{x,y}}$. This is what allows us to obtain \eqref{eq:selfimprov}.
\end{remark}

\subsection{Step 4: End of the proof of Theorem \ref{prop:bvorGga}.}
We can now proceed to conclude the proof of the main Theorem, putting the several pieces together.
\begin{lemma}\label{prop:estGga}
    Fix $\a \geq \a_0$ and let $G_{g,\a}$ be the Green's function for the operator $(\Dg+\a)^k$ in $M$ defined in Proposition \ref{prop:defGga}. Then for all $\epsilon \in (0,1)$, there exists $C_\epsilon>0$ such that for all $x \neq y$, we have
    \begin{equation}\label{eq:estGga}
        \abs{\Gga(x,y)} \leq C_\epsilon\begin{cases}
        \dg{x,y}^{-(n-2k)} & \sqa \dg{x,y} \leq 1\\
        \dg{x,y}^{-(n-2k)} e^{-(1-\epsilon)\sqa \dg{x,y}} & \sqa \dg{x,y} \geq 1\\
        e^{-(1-\epsilon)\sqa \inj/2} & \dg{x,y} \geq \inj/2
    \end{cases}.
        \end{equation}
\end{lemma}
\begin{proof}
    Fix any $0 < \epsilon < 1$. First, let $x,y \in M$ be such that $\dg{x,y} \leq 1/\sqa$ and $x\neq y$. Going back to the definition of $\Gga$ in \eqref{def:Gga}, we use \eqref{eq:estGg}, \eqref{eq:estTGsqa} and the fact that $u_x \in C^0(M)$ satisfies \eqref{eq:1stestux}, and we obtain 
    \[  \abs{\Gga(x,y)} \leq C \dg{x,y}^{-(n-2k)}
        \]
    for all $\dg{x,y} \leq 1/\sqa$.
    \par Now let $1\sqa \leq \dg{x,y} < \inj/2$. We use \eqref{eq:estGg} and \eqref{eq:estTG}, there is a constant $C_\epsilon$ independent of $\a$, $x,y$ such that
    \[  \abs{\Gg^i(x,y)} \leq C_\epsilon \a^{-i} \dg{x,y}^{-(n-2k)}e^{-(1-\epsilon)\sqa \dg{x,y}}
        \]
    for $i= 0,\ldots N-1$, writing $\Gg^0 := \Gg$. Up to taking a slightly smaller $0 < \epsilon' < \epsilon$, we also have for all $y \in M$
    \begin{align*}  
        \abs{u_{\a,x}(y)} &\leq C_{\epsilon'} \a^{-k} e^{-(1-\epsilon')\sqa \dg{x,y}}\\
            &\leq C_{\epsilon,\epsilon'} \a^{-\frac{n}{2}}\dg{x,y}^{-(n-2k)}e^{-(1-\epsilon)\sqa \dg{x,y}},
    \end{align*}
    using \eqref{eq:estux}.
    We have thus obtained 
    \[  \abs{\Gga(x,y)} \leq C_\epsilon \dg{x,y}^{-(n-2k)} e^{-(1-\epsilon)\sqa\dg{x,y}}.
        \]
    Finally, when $\dg{x,y} \geq \inj/2$, the estimate follows from \eqref{eq:estux}.
\end{proof}

\begin{lemma}\label{prop:Ggasym}
    Fix $\a > \a_0$ and let $\Gga$ be the Green's function for the operator $(\Dg+\a)^k$ in $M$ defined in Proposition \ref{prop:defGga}. Then for all $x\neq y$ in $M$, $\Gga(x,y) = \Gga(y,x)$.
\end{lemma}
\begin{proof}
    Let $\phi \in C^\infty(M)$, define for $y \in M$
    \[  h(y) := \intM{\Gga(z,y) \phi(z)}(z).
        \]
    Then since $\Gga(\cdot, y)\in C^\infty(M \setminus \{y\}) \cap L^1(M)$, $h$ is well-defined. We claim that $h$ is continuous on $M$. To prove the latter, fix $y\in M$ and take a sequence of points in $M$, $y_m \to y$ as $m\to \infty$. Let $\delta_m := \dg{y_m,y}$, we have 
    \[  h(y_m) = \intM[\Bal{y_m}{\delta_m/2}]{\Gga(z,y_m)\phi(z)}(z) + \intM[M\setminus\Bal{y_m}{\delta_m/2}]{\Gga(z,y_m)\phi(z)}(z).
        \]
    On the one hand, by \eqref{eq:estGga}, we have 
    \[  \intM[\Bal{y_m}{\delta_m/2}]{G(y_m,z)\phi(z)}(z) = o(1)
        \]
    as $m \to \infty$. On the other hand, for $z \in M\setminus \Bal{y_m}{\delta_m/2}$, we have $\dg{z,y_m} \geq \frac{1}{3}\dg{z,y}$, so that using \eqref{eq:estGga}, 
    \[  \abs{G(y_m,z)\phi(z)} \leq C \dg{z,y}^{2k-n}.
        \]
    By dominated convergence, we conclude that 
    \begin{align*}  
        \lim_{m\to \infty} h(y_m) &= \lim_{m\to \infty} \intM[M\setminus\Bal{y_m}{\delta_m/2}]{\Gga(z,y_m)\phi(z)}(z)\\
            &= \intM{\Gga(z,y)\phi(z)}(z),
        \end{align*}
    and $h$ is continuous at $y\in M$.
    \par Let now $g \in C^\infty(M)$ be the unique solution to $(\Dg+\a)^k g = \phi$ on $M$. Since $\Gga$ is the Green's function for the operator $(\Dg+\a)^k$ in $M$, we have, for all $x\in M$,
    \[  g(x) = \intM{\Gga(x,y) \phi(y)}(y).
        \]
    Now since $h\in C^0(M)$, $g\in C^\infty(M)$, one has $h-g \in L^\infty(M)$ and there is a unique $\psi_0 \in \Sob(M)$ such that
    \[  (\Dg+\a)^k \psi_0 = h - g \quad \text{in the weak sense on $M$.}
        \]
    Moreover, elliptic theory gives $\psi_0 \in \Sob[2k,p](M)$ for all $p\geq 1$. By Fubini's theorem, and since $\Gga[x] \in L^1(M)$ by Proposition \ref{prop:defGga}, we have
    \begin{align*}
        \intM{h\, (\Dg+\a)^k \psi_0} &= \intM{\bpr{\intM{\Gga(x,y)\phi(x)}(x)}(\Dg+\a)^k \psi_0(y)}(y)\\
            &= \intM{\bpr{\intM{\Gga(x,y)(\Dg+\a)^k \psi_0(y)}(y)}\phi(x)}(x)\\
            &= \intM{\psi_0(x) \phi(x)}(x).
    \end{align*}
    Use the definition of $g$ and integrate by parts: Since $\psi_0 \in \Sob[2k,p](M)$ for any $p\geq 1$,
    \[
        \intM{\psi_0 \phi} = \intM{\psi_0 (\Dg+\a)^k g} = \intM{(\Dg+\a)^k \psi_0 g}.
    \]
    By definition of $\psi_0$, we have thus shown in the end that
    \[  \intM{(\Dg+\a)^k \psi_0 \big(h-g\big)} = 0 = \intM{\big(h-g\big)^2},
        \]
    so that $h(x) = g(x)$ for almost every $x \in M$, and thus $h \equiv g$ on $M$ by continuity of $h, g$. This shows that for all $\phi \in C^\infty(M)$, and all $x\in M$,
    \[  \intM{\bpr{\Gga(x,y)- \Gga(y,x)}\phi(y)}(y) = 0.
        \]
    Since $\Gga(x, \cdot),\, \Gga(\cdot,x) \in L^1(M)$ for all $x \in M$, we deduce that 
    \[  \Gga(x,y) = \Gga(y,x)
        \] 
    for almost every $y \in M$, and we conclude with the continuity of $\Gga$ in $M\times M \setminus\Diag$.
\end{proof}

\begin{proof}(of Theorem \ref{prop:bvorGga})
    Recall the definition \eqref{def:Gga} of $\Gga$, we have
    \[  \Gga(x,y) = \Gg^*(x,y) + u_{\a,x}(y),
        \]
    where $\Gg^*$ was defined in \eqref{def:Gg*}, and $u_{\a,x}$ is the unique solution to \eqref{eq:equx}. It is a Green's function for $(\Dg + \a)^k$ in $M$ as we proved in Proposition \ref{prop:defGga}, and we have proved the estimates on $\Gga$ in Lemma \ref{prop:estGga}. Moreover, when $\dg{x,y} \leq 1/\sqa$, by Proposition \ref{prop:defgamx}, $\Gg^*(x,y)$ satisfies \eqref{eq:estG*}, so that $\Gga(x,y)$ satisfies \eqref{eq:bvorGg}, since $u_{\a,x}$ is bounded.
    \par Uniqueness was proved in Lemma \ref{prop:uniqpos}, while the symmetry was proved in Lemma \ref{prop:Ggasym}. It remains only to show the positivity of the Green's function. 
    \par Let $H_{g,\a}^{(1)}$ be the Green's function for the operator $\Dg + \a$ in $M$ as constructed in Proposition \ref{prop:defGga}, and define for all $x\neq y$, as in \eqref{def:Gok}, 
    \[  H_{g,\a}^{(k)}(x,y) := \intM{H_{g,\a}^{(k-1)}(x,z) H_{g,\a}^{(1)}(z,y)}(z)
        \]
    for $k\geq 2$. This function is well-defined provided that $n>2k$, and smooth away from the diagonal, thanks to Lemma \ref{prop:rexpGir}. The same argument as in \eqref{eq:gkisgreen} shows that $H_{g,\a}^{(k)}$ is a Green's function for the operator $(\Dg+\a)^k$ in $M$. By Lemma \ref{prop:uniqpos}, $H_{g,\a}^{(k)} = \Gga$ in $M\times M \setminus \Diag$.
    Now, recall that $H_{g,\a}^{(1)}$ is positive by the strong maximum principle and Hopf Lemma (see \cite[Theorem 3.]{Rob10}). Thus, $\Gga= H_{g,\a}^{(k)}$ is positive in $M\times M \setminus \Diag$.
\end{proof}
\begin{remark}
    Note that the factorized form of $(\Dg+\a)^k$ is crucial to obtain positivity. In general, the question of positivity of for higher-order operators is a hard problem. We refer to \cite{GazGruSw10,GruRob10, GruSw96} for positivity results for polyharmonic operators on domains of $\R^n$ with specific boundary conditions. The argument that we describe here for the positivity would work for any polyharmonic operator which is decomposed as the product of coercive operators of order 2, such as the GJMS operator in a Riemannian manifold with Einstein metric (see \cite{FefGra12}).
\end{remark}

\subsection{Control on the derivatives.}
We are now interested in estimates similar to \eqref{eq:estGga} for the derivatives of $\Gga$. 
Let $f \in C^\infty(M)$, then the $l^{\text{th}}$ covariant derivative $\nabla^l f \in (T^*M)^l$ is a $(l,0)$-tensor. 
Fix $x \in M$, on a neighborhood $\Bal{x}{\tau}$ with $\tau < \inj$, the exponential map at $x$ allows us to define the metric on $\Bal{0}{\tau} \subset \R^n$,
\begin{equation}\label{def:gTil}
    \Tilde{g}(u) := \exp_x^* g(u).
    \end{equation}
Since $M$ is compact, there is a global constant $C>0$ independent of $x\in M$ such that for $u \in \Bal{0}{\tau}$,
\begin{equation}\label{eq:rimder}
    \begin{bigcases}
        &\abs{\Tilde{g}(u)_{ij} - \delta_{ij}} \leq  C\abs{u}^2 & 1 \leq i,j \leq n\\
        &\abs{\nabla_{\Tilde{g}}^l\Tilde{f}(u) - \nabla_\xi^l \Tilde{f}(u)} \leq C \bpr{\abs{u}\abs{\nabla_{\xi}^{l-1} \Tilde{f}(u)} + \sum_{m=1}^{l-2} \abs{\nabla_{\xi}^m \Tilde{f}(u)} }\hspace{-1em} & l\geq 2
    \end{bigcases}
\end{equation}
where $\Tilde{f} := f\circ \exp_x$, and $\Tilde{g}(u)_{ij}$ are the components of $\Tilde{g}$ at $u$ in the exponential chart.

\begin{proposition}\label{prop:bvorDux}
    Let $\a_0\geq 1$ be given by Proposition \ref{prop:bvorux}. For all $0<\epsilon <1$, there is a constant $C_\epsilon>0$ such that for all $\a\geq \a_0$, for all $x, y \in M$ and for $l=1, \ldots 2k-1$, we have
    \[  \abs{\nabla^l u_{\a,x}}_{g}(y) \leq C \a^{-k + \frac{l}{2}} \begin{cases}
        1 & \sqa \dg{x,y} \leq 1\\
        e^{-(1-\epsilon)\sqa \dg{x,y}} & \sqa \dg{x,y} \geq 1\\
        e^{-(1-\epsilon)\sqa \inj/2} & \dg{x,y} \geq \inj/2
    \end{cases}.
        \]
\end{proposition}
\begin{proof}
    Let us fix $x\in M$ and let $y \in M$. We prove the estimates on the derivatives of $u_{\a,x}$ by using elliptic estimates in balls centered at $y$. Observe that for $z\in \Bal{y}{1/\sqa}$,
    \[  \dg{x,y} -1/\sqa \leq \dg{x,z} \leq \dg{x,y} + 1/\sqa.
        \]
    With this realization, and with Proposition \ref{prop:bvorux}, we obtain
    \begin{equation}\label{tmp:estuxy}
    \begin{aligned}  \abs{u_{\a,x}(z)} &\leq C_\epsilon \a^{-k} \begin{cases}
        1 & \sqa \dg{x,y} \leq 1\\
        e^{-(1-\epsilon)\sqa \dg{x,y}} & \sqa \dg{x,y} \geq 1\\
        e^{-(1-\epsilon)\sqa \inj/2} & \dg{x,y} \geq \inj/2
    \end{cases}\\
        &= C_\epsilon \a^{-k} \Pea(x,y)
        \end{aligned}
    \end{equation}
    for all $z\in \Bal{y}{1/\sqa}$, where $\Pea$ is as in \eqref{eq:defPea}. We also have by \eqref{eq:estgamx} that there exists a constant $C_\epsilon'>0$ independent of $x,y$ and $\a\geq \a_0$ such that 
    \begin{equation}\label{tmp:estgamxy}
        \abs{\gamma_{\a,x}(z)} \leq C_\epsilon' \a^{-N+\frac{n}{2}} \Pea(x,y)
        \end{equation}
    for all $z\in\Bal{y}{1/\sqa}$. 
    \par Set $v := \sqa \exp^{-1}_y(z)$, then if $z \in \Bal{y}{1/\sqa}$, $v\in \Bal{0}{1} \subset \R^n$. For $\a\geq \a_0$, we have $1/\sqa < \inj$, define
    \[  \Tilde{g}_{\a,y}(v) := \exp_y^* g(v/\sqa) \qquad \text{and} \qquad \begin{cases}
        \Tilde{u}_{\a,x,y}(v) := u_{\a,x}(\exp_y(v/\sqa)) \\
        \Tilde{\gamma}_{\a,x,y}(v) := \gamma_{\a,x}(\exp_y(v/\sqa))
    \end{cases}.
        \]
    Observe that $\Tilde{\gamma}_{\a,x,y} \in C^0(\Bal{0}{1})$ since $\gamma_{\a,x} \in C^0(M)$, and that $\Tilde{u}_{\a,x,y} \in L^\infty(\Bal{0}{1})$. We compute, for all $v\in \Bal{0}{1}$,
    \[  \Dg u_{\a,x}(\exp_y(v/\sqa)) = \a \Dg[\Tilde{g}_{\a,y}]\Tilde{u}_{\a,x,y}(v)
        \]
    so that $\Tilde{u}_{\a,x,y}$ solves the following equation on $\Bal{0}{1}$,
    \[  \a^k (\Dg[\Tilde{g}_{\a,y}] + 1)^k \Tilde{u}_{\a,x,y} = \Tilde{\gamma}_{\a,x,y}.
        \]
    Using the fact that when $\a \to \infty$, $\Tilde{g}_{\a,y} \to \xi$ the Euclidean metric in $C^{\infty}_{loc}(\R^n)$, $(\Dg[\Tilde{g}_{\a,y}]+1)^k$ is an elliptic operator with coefficients bounded independently of $\a \geq \a_0$. Elliptic theory gives $\Tilde{u}_{\a,x,y} \in C^{2k-1, \theta}(\Bal{0}{1/2})$ for all $\theta \in (0,1)$. There is a constant $C>0$ that does not depend on $\a$, $x,y$, such that for $l=1,\ldots 2k-1$ and all $v \in \Bal{0}{1/2}$ we have
    \begin{align*}
        \abs{\nabla^l \Tilde{u}_{\a,x,y}(v)} &\leq C (\norm{\Tilde{u}_{\a,x,y}}_{L^\infty(\Bal{0}{1})} + \norm{\a^{-k}\Tilde{\gamma}_{\a,x,y}}_{C^{0}(\Bal{0}{1})})\\
            &\leq C \a^{-k} \Pea(x,y),
    \end{align*}
    using \eqref{tmp:estuxy} and \eqref{tmp:estgamxy}.
    Note that here the gradient and the norm are taken in the Euclidean space $\R^n$. To get the metric-related quantities we use \eqref{eq:rimder}, for all $v \in \Bal{0}{1/2}$, we have
    \begin{align*}
        \abs{\nabla_g^l u_{\a,x}}_g(\exp_y(v/\sqa)) &\leq C\abs{\nabla_{\Tilde{g}_{\a,y}}^l \Tilde{u}_{\a,x,y}(v)} \leq C\sum_{m=1}^{l} \a^{\frac{m}{2}} \abs{\nabla^m \Tilde{u}_{\a,x,y}(v)}\\
            &\leq C\a^{\frac{l}{2}-k} \Pea(x,y).
        \end{align*}
    Taking this inequality at $v = 0 \in \Bal{0}{1/2}$, then $\exp_y(0) = y$, and we conclude.
\end{proof}

The following Proposition extends Corollary \ref{prop:betbvorDGak} to the Riemannian case, showing that the pointwise decomposition of $\Gga$ in \eqref{eq:bvorGg} can be differentiated formally. 

\begin{proposition}\label{prop:bvorDgGga}
    Fix $0<\epsilon<1$, there exists $\a_0\geq 1$ such that for all $\a \geq \a_0$, the derivatives of $\Gga$ have the following estimates. There exists a constant $C>0$ independent of $\a$ such that for all $x\neq y \in M$, and for $l = 1,\ldots 2k-1$, $x\neq y$, 
    \[  \abs{\nabla^l \Gga[x](y)}_g \leq C \begin{cases}
        \dg{x,y}^{-(n-2k + l)}  & \sqa \dg{x,y} \leq 1\\
        \dg{x,y}^{-(n-2k + l)} e^{-(1-\epsilon)\sqa \dg{x,y}} & \sqa\dg{x,y} \geq 1\\
        e^{-(1-\epsilon)\sqa \inj/2} & \dg{x,y} \geq \inj/2
    \end{cases}.
        \]
    Moreover, for $\dg{x,y} \leq 1/\sqa$ with $x\neq y$,
    \[  \abs{\nabla^l \bpr{\dg{x,y}^{n-2k} \Gga[x](y)}}_g \leq C\dg{x,y}^{-l}\eta(\sqa\dg{x,y}),
        \]
    where $\eta$ is defined in \eqref{def:eta}.
\end{proposition}
\begin{proof}
    We estimate the derivative of each term in the expression \ref{def:Gga} independently.
    \par First, we compute estimates for the derivatives of $\Gg$. Fix $x\in M$, and let $u := \exp_x^{-1}(y)$, then by the definition \eqref{def:Gg}, and \eqref{eq:estGg0}, we have for $u \in \Bal{0}{\tau_0}$,
    \begin{equation}\label{tmp:estDGg}
    \begin{aligned}
        \abs{\nabla_g^l \Gg[x]}_g (\exp_x(u)) &\leq C \abs{\nabla_{\Tilde{g}}^l \bpr{\chi(\abs{u}) \Ga[0]{k}(u)}}\\
            &\leq C \sum_{m=1}^{l} \abs{\nabla_{\xi}^m \bpr{\chi(\abs{u}) \Ga[0]{k}(u)}} \\
            &\leq C \begin{cases}
                \abs{u}^{-(n-2k+l)} & \sqa \abs{u} \leq 1\\
                \a^{k\frac{n-3}{4} + \frac{l}{2}} \abs{u}^{\frac{n(k-2)+k}{2}} e^{-\sqa \abs{u}} & \sqa \abs{u} \geq 1
            \end{cases},
    \end{aligned}
\end{equation}
    using \eqref{eq:rimder}, where $\Tilde{g}$ is as in \eqref{def:gTil}.
    By the multiplication by the cutoff, we also have 
    \[ \abs{\nabla^l \Gg[x](y)}_g = 0 \qquad  \text{for $y \in M\setminus\Bal{x}{\tau_0}$.}
    \]
    \par In a second step, we use \eqref{eq:defggi} and the fact that $\big|\nabla^l \Gg[x]\big|_g \in L^1(M)$ for $l=1,\ldots 2k-1$ thanks to \eqref{tmp:estDGg}, and we obtain 
    \[  \abs{\nabla^l \Gg[x]^i(y)}_g \leq \intM{ \abs{\Gam{i}(x,z)}\abs{\nabla^l\Gg[z](y)} }(z)_g, 
        \]
    for $i=1,\ldots N-1$.
    Using Lemma \ref{prop:rexpGir}, we then have, when $x\neq y$ with $\sqa \dg{x,y} \leq 1$,
    \begin{equation}\label{tmp:1estGi}
        \abs{\nabla^l \Gg[x]^i(y)}_g \leq C_i \begin{cases}
        \dg{x,y}^{-(n-2k - 2i + l)} & \text{when $2k+2i-l < n$}\\
        1 + \abs{\log(\sqa \dg{x,y})} & \text{if $2k+2i - l = n$}\\
        \a^{-\frac{2k+2i-n-l}{2}} & \text{when $2k+2i -l > n$}
    \end{cases}.
        \end{equation}
    We also have, when $\sqa \dg{x,y} \geq 1$,
    \[  \abs{\nabla^l \Gg[x]^i(y)}_g \leq C \a^{p_i + \frac{l}{2}} \dg{x,y}^{\rho_i} e^{-\sqa \dg{x,y}}
        \]
    where $p_i,\rho_i$ were defined in \eqref{eq:defpiri}. Finally, $\abs{\nabla^l \Gg[x]^i(y)}_g = 0$ when $\dg{x,y} \geq (i+1)\tau_0$, where $\tau_0$ is as in Proposition \ref{prop:defgamx}. There is then a constant $C_\epsilon>0$ such that for $\a$ large enough and all $x,y$ with $\sqa \dg{x,y} \geq 1$, 
    \[  \abs{\nabla^l \Gg[x]^i(y)}_g \leq C_\epsilon \a^{-i} \dg{x,y}^{-(n-2k + l)} e^{-(1-\epsilon)\sqa\dg{x,y}}
        \]
    for $i = 1,\ldots N-1$, $l = 1,\ldots 2k-1$. 
    \par For the last term $u_{\a,x}$, choose $0<\epsilon' < \epsilon$, Proposition \ref{prop:bvorDux} gives
    \begin{multline}\label{tmp:2estu}
        \abs{\nabla^l u_{\a,x}(y)}_g \leq C \a^{-k+\frac{l}{2}} \begin{cases}
            1 & \sqa \dg{x,y} \leq 1\\
            e^{-(1-\epsilon')\sqa\dg{x,y}} & \sqa \dg{x,y} \geq 1\\
            e^{-(1-\epsilon')\sqa \inj/2} & \dg{x,y} \geq \inj/2
        \end{cases} \\
        \leq C\a^{-\frac{n}{2}} \dg{x,y}^{-(n-2k+l)} \begin{cases}
            1 & \sqa \dg{x,y} \leq 1\\
            e^{-(1-\epsilon)\sqa\dg{x,y}} & \sqa \dg{x,y} \geq 1\\
            e^{-(1-\epsilon)\sqa \inj/2} & \dg{x,y} \geq \inj/2
        \end{cases}.
    \end{multline}
    Putting this together, we obtain for $l=1, \ldots 2k-1$,
    \begin{equation}\label{eq:DGgasum}\begin{aligned}
        \abs{\nabla^l \Gga[x](y)}_g &\leq \abs{\nabla^l \Gg[x](y)}_g + \sum_{i=1}^{N-1} \abs{\nabla^l \Gg[x]^i(y)}_g + \abs{\nabla^l u_{\a,x}(y)}_g\\
            &\leq C \begin{cases}
        \dg{x,y}^{-(n-2k + l)}  & \sqa \dg{x,y} \leq 1\\
        \dg{x,y}^{-(n-2k + l)} e^{-(1-\epsilon)\sqa \dg{x,y}} & \sqa\dg{x,y} \geq 1\\
        e^{-(1-\epsilon)\sqa \inj/2} & \dg{x,y} \geq \inj/2
    \end{cases}
    \end{aligned}\end{equation}
    which proves the first part of Proposition \ref{prop:bvorDgGga}.
    \par For the second part of the proof, we use again \eqref{def:Gga}, for $\dg{x,y} \leq 1/\sqa$ we have
    \begin{multline}\label{tmp:decompF}
        \abs{\nabla^l \bpr{\dg{x,y}^{n-2k}\Gga[x](y)}}_g\\ = \abs{\nabla^l \bpr{\dg{x,y}^{n-2k}\Gg[x](y)}}_g + \bigO\bpr{\tfrac{1}{\a}\eta\big(\sqa \dg{x,y}\big)\dg{x,y}^{-l}}
    \end{multline}
    using \eqref{tmp:1estGi} and \eqref{tmp:2estu}. Let $F_x(y) := \dg{x,y}^{n-2k}\Gg[x](y)$ for $y\in \Bal{x}{1/\sqa}$, and $\Tilde{F}(u) := F_x(\exp_x(u)) = \abs{u}^{n-2k} \Ga[0]{k}(u)$ for $\abs{u}\leq 1/\sqa$. Then, writing $\Tilde{g} := \exp_x^* g$ as before,
    \begin{multline*}
        \abs{\nabla_g^l F_x}_g(\exp_x(u)) = \abs{\nabla_{\Tilde{g}}^l \Tilde{F}(u)} (1 + \bigO(\abs{u}^2))\\
            = \bpr{\abs{\nabla_\xi^l \Tilde{F}(u)} + \bigO\bpr{\abs{u} \abs{\nabla_\xi^{l-1} \Tilde{F}(u)}} + \bigO\bpr{\sum_{m=1}^{l-2} \abs{\nabla_\xi^{m}\Tilde{F}(u)}}}(1 + \bigO(\abs{u}^2))
    \end{multline*}
    for all $\abs{u} \leq 1/\sqa$, using \eqref{eq:rimder}. Each term in this sum is estimated using Corollary \ref{prop:betbvorDGak}, and we finally obtain in \eqref{tmp:decompF} that, when $\sqa\dg{x,y} \leq 1$ and $\a\geq \a_0$,
    \[
        \abs{\nabla^l \bpr{\dg{x,y}^{n-2k}\Gga[x](y)}}_g \leq C \eta\big(\sqa\dg{x,y}\big) \dg{x,y}^{-l}.
    \]
\end{proof}

\subsection{Mass of the operator in dimension $n = 2k+1$.}
We conclude this paper by a consideration in the case where $n=2k+1$. When $n=2k+1$, the Green's function $\Gga$ can be re-written as
\[  \Gga(x,y) = \frac{c_{n,k}}{\dg{x,y}} + \mu_x(y),
    \]
where $c_{n,k}$ is given by \eqref{def:cnk}, and $\mu_x(y)$ is a continuous function for all $x\in M$, as recalled below. It is then standard to define the \emph{mass} of the operator as the quantity $\mu_x(x)$, see \cite{SchYau81} for the conformal Laplacian, and \cite{HumRau09} for the Paneitz-Branson operator. 
\begin{lemma}
    There exists $\a_0 \geq 1$ and $C_1,C_2>0$ such that for all $x \in M$, $\a \geq \a_0$, 
    \[  C_1 \sqa \leq -\mu_x(x) \leq C_2 \sqrt{\a}.
        \]
\end{lemma}
\begin{proof}
    First, by relation \eqref{eq:bvorGg} we obtain an upper bound on the absolute value of the mass : There exists $\a_0\geq 1$ and $C>0$ such that for all $x\in M$, $\a\geq \a_0$,
    \[   \abs{\mu_x(x)} \leq C \sqa.
        \] 
    Now for the second part, fix $\a_0\geq 1$ given by Theorem \ref{prop:bvorGga}. From the decomposition \eqref{def:Gga}, one obtains that 
    \[  \abs{\mu_x(y)} \leq \abs{\Gg(x,y) - \frac{c_{n,k}}{\dg{x,y}}} + \sum_{i=1}^{N} \abs{\Gg^i(x,y)} + \abs{u_{\a,x}(y)},
        \]
    so that 
    \[  \mu_x(y) = \Gg(x,y) - \frac{c_{n,k}}{\dg{x,y}} + \bigO(\a^{-1/2})
        \]
    using \eqref{eq:estTGsqa}, \eqref{eq:estTG} and Proposition \ref{prop:bvorux}. Now with the definition \eqref{def:Gg}, we have for $\dg{x,y} \leq \tau_0/2$ given in Proposition \ref{prop:defgamx},
    \begin{equation}\label{tmp:mas1}
        \mu_x(y) = \bpr{\Ga{k}(0,u) - \frac{c_{n,k}}{\abs{u}}}\at{u = \exp_{x}^{-1}(y)} + \bigO(\a^{-1/2}).
    \end{equation}
    Using the notations of Proposition \ref{prop:excbvorGa}, recall that  
    \[  \Ga{k}(0,u) - \frac{c_{n,k}}{\abs{u}} = R_\a(0,u),
        \]
    where $R_{\a}$ satisfies \eqref{eq:repforR},
    \begin{equation*}
        R_\a(0,u) = - \int_{\R^n} c_{n,k} \abs{u-z}^{-1} \sum_{l=0}^{k-1} \a^{k-l}\Dg[\xi]^l \Ga[0]{k}(z) dz
        \end{equation*}
    for all $u \neq 0$. By integration by part, and using Proposition \ref{prop:bvorDGak}, we obtain
    \begin{equation}\label{tmp:eqRa}
        R_\a(0,u) = -c_{n,k}\sum_{l=0}^{k-1} \a^{k-l}\int_{\R^n} (\Dg[\xi]^l)_{(z)} \bpr{\abs{u-z}^{-1}}\, \Ga[0]{k}(z) dz.
        \end{equation}
    Simple calculations show that for $l=0,\ldots k-1$, there is a constant $d_{n,k}>0$ depending only on $n,k$ such that
    \[  \Dg[\xi]^l \big(\tfrac{1}{\abs{u}}\big) = d_{n,k}\frac{1}{\abs{u}^{1+2l}}.
        \]
    Since $\Ga{k} >0$, \eqref{tmp:eqRa} gives 
    \[
        -R_\a(0,u) \geq \Tilde{c}_{n,k}\, \a\int_{\R^n}\frac{1}{\abs{u-z}^{2k-1}} \Ga[0]{k}(y)\, dy.
    \]
    Let $\delta \in (0,1)$ whose value will be fixed later. Proposition \ref{prop:excbvorGa} shows that there exists a constant $C>0$ independent of $\a$, such that 
    \[  \Ga{k}(x,y) \geq \frac{c_{n,k}}{\abs{x-y}} (1 - C \sqrt{\a}\abs{x-y})
        \]
    for all $\abs{x-y} \leq 1/\sqrt{\a}$. Thus, 
    \begin{align*}  
        -R_\a(0,u) &\geq \Tilde{c}_{n,k} \a \int_{\Bal{0}{\delta/\sqrt{\a}}} \frac{1}{\abs{y-u}^{2k-1}} \frac{1}{\abs{y}} (1 - C \delta) dy\\
            &\geq \Tilde{c}_{n,k} \a \frac{\delta}{\sqrt{\a}} (1 - \delta C)
    \end{align*}
    for all $\abs{u} \leq \tfrac{2\delta}{\sqa}$. Finally, fix $\delta>0$ small enough so that $1-\delta C \geq \tfrac{1}{2}$, we obtain that there exists $C>0$ independent of $\alpha$ such that for all $\abs{u} \leq \tfrac{2\delta}{\sqa}$,
    \begin{equation}\label{tmp:mas2}
        -R_\a(0,u) \geq C \sqrt{\a}.
    \end{equation}
    The continuity of $R_\a$ follows from elliptic theory, since $R_\a$ satisfies 
    \[  \Dg[\xi]^k R_{\a,x} = h_{\a,x}
        \]
    and $h_{\a,x} \in L^p(\R^n)$ for all $p>\frac{n}{n-2}$ using \eqref{tmp:esth}, with $n-2 < 2k$. This implies the continuity of $\mu_x(y)$, since all the terms $\Gg^{i}$ and $u_{\a,x}$ are also continuous by Lemma \ref{prop:rexpGir}. Putting \eqref{tmp:mas1} and \eqref{tmp:mas2} together and evaluating them at $u=0$ gives
    \[  \mu_x(x) = R_\a(0,u)\at{u = 0} + \bigO(\a^{-1/2}) \leq -C_1 \sqrt{\a}
        \]
    for some $C_1>0$ and for $\a \geq \a_0$. In particular, $\mu_x(x) \to -\infty$ as $\a \to \infty$.
\end{proof}
Note that the terms due to the presence of the metric contribute only in $\bigO(\a^{-1/2})$ whereas the main contribution leading to the divergence of the mass comes from the lower-order terms in the operator $(\Dg[\xi] +\a)^k$ on the Euclidean space.

\appendix
\section{Generalized Giraud's Lemmas}\label{sec:giraud}
In this Appendix, we compute some convolution estimates that are needed in the previous Sections to obtain the bounds on the Green's functions. These are modified versions of results known as Giraud's Lemma in the literature (see \cite[Proposition 4.12]{Aubin82}, \cite[Lemma 7.5]{Heb14}), proved in \cite[p. 150]{Gir29}.
\begin{lemma}[Exponential Giraud's]\label{prop:expGir}
    Let $X, Y \in C^0(\R^n \times \R^n \setminus \Diag)$. Assume that there exist $\beta, \gamma \in (0,n]$, and $\rho, \nu > -n$ such that
    \begin{align*}  
        \abs{X(x,y)} &\leq \begin{cases}
        \abs{x-y}^{\beta-n} & \text{if }\abs{x-y} \leq 1\\
        \abs{x-y}^\rho e^{-\abs{x-y}} & \text{if } \abs{x-y} \geq 1
        \end{cases},  \\
        \abs{Y(x,y)} &\leq \begin{cases}
        \abs{x-y}^{\gamma-n} & \text{if } \abs{x-y} \leq 1\\
        \abs{x-y}^\nu e^{-\abs{x-y}} & \text{if } \abs{x-y} \geq 1
    \end{cases}
        \end{align*}
    for all $x\neq y$. Let $Z(x,y) := \int_{\R^n} X(x,z) Y(z,y)\, dz$ for $x\neq y$. Then $Z \in C^0(\R^n \times \R^n \setminus \Diag)$ and there exists $C>0$ such that for all $x\neq y$ :
    \begin{itemize}
        \item If $\abs{x-y} \leq 1$,
        \[  \abs{Z(x,y)} \leq \begin{cases}
            C\abs{x-y}^{\beta+ \gamma-n} & \text{when } \beta+\gamma < n\\
            C\big(1 + \abs{\log \abs{x-y}}\big) & \text{when } \beta+\gamma = n\\
            C & \text{when } \beta+\gamma > n.
            \end{cases}
            \]
        \item If $\abs{x-y} \geq 1$,
        \[ \abs{Z(x,y)} \leq C \abs{x-y}^{\rho + \nu + n} e^{-\abs{x-y}}.
            \]
    \end{itemize}
    Moreover, when $\beta+\gamma > n$, $Z$ is continuous on the whole $\R^n \times \R^n$.
\end{lemma}
The following Proposition extends this result on a compact Riemannian manifold $M$ with injectivity radius $\inj>0$, and with a scale parameter $\a\geq 1$.
\begin{lemma}\label{prop:rexpGir}
    Let $(M,g)$ be a closed Riemannian manifold, and $\tau, \sigma>0$ such that $\tau + \sigma < \inj$. Let $X, Y \in C^0(M\times M \setminus \Diag)$, such that for all $x,y \in M$, $X(x,\cdot)$ is supported in $\Bal{x}{\tau}$ and $Y(\cdot,y)$ in $\Bal{y}{\sigma}$. Assume that there exist $\beta, \gamma \in (0,n]$, $p, q \geq 0$ and $\rho, \nu > -n$ satisfying
    \begin{equation}\label{eq:tmpcondpro}
    \begin{cases}
        2p - \rho \leq n-\beta\\
        2q - \nu \leq n-\gamma
    \end{cases}
        \end{equation}
    and such that, for all $x\neq y$, $\a \geq 1$,
    \begin{align*}
        \abs{X(x,y)} &\leq \begin{cases}
        \dg{x,y}^{\beta-n} &  \text{if } \sqa \dg{x,y} \leq 1\\
        \a^p\dg{x,y}^\rho \expab{x,y} & \text{if } \sqa \dg{x,y} \geq 1
    \end{cases}\\
    \abs{Y(x,y)} &\leq \begin{cases}
        \dg{x,y}^{\gamma-n} & \text{if } \sqa \dg{x,y} \leq 1\\
        \a^q \dg{x,y}^\nu \expab{x,y} & \text{if } \sqa \dg{x,y} \geq 1.
    \end{cases}
        \end{align*}
    Let, for all $x\neq y$, $Z(x,y) := \intM{X(x,z)Y(z,y)}(z)$. Then $Z \in C^0(M\times M \setminus \Diag)$ and, for all $x\in M$, $Z(x,\cdot)$ is supported in $\Bal{x}{\tau +\sigma}$. There exists $\a_0\geq 1$ and $C>0$ such that for all $x\neq y$ and $\a \geq \a_0$, we have the following :
    \begin{itemize}
        \item If $\sqa\dg{x,y} \leq 1$,
        \[  \abs{Z(x,y)} \leq C \begin{cases}
            \dg{x,y}^{\beta + \gamma - n} & \text{when }\beta + \gamma < n\\
            \bpr{1 + \abs{\log (\sqa \dg{x,y})}} & \text{when }\beta + \gamma = n\\
            \a^{-\frac{\beta+\gamma-n}{2}} & \text{when } \beta + \gamma > n.
        \end{cases}
            \]
        \item If $\sqa \dg{x,y} \geq 1$,
        \[  \abs{Z(x,y)} \leq C \a^{n-\frac{\beta+\gamma}{2} + \frac{\rho + \nu}{2}} \dg{x,y}^{\rho + \nu + n} e^{-\sqa \dg{x,y}}.
            \]
    \end{itemize}
    Moreover, when $\beta + \gamma > n$, $Z \in C^0(M\times M)$.
\end{lemma}
Since the proofs of these two Lemmas are the same, we only show the second one. The proof of Lemma \ref{prop:expGir} is identical setting formally $\tau,\sigma = \infty$ and $\a = 1$, and taking the integrals on the Euclidean space.

\begin{proof}
    Up to choosing $\a_0$ large enough, we can assume that $5/\sqa < \tau + \sigma < \inj$.
    \par For the first part of the proof, let $x,y \in M$ with $x\neq y$, and assume that $\dg{x,y} \leq 2/\sqa$. We have 
    \begin{multline}\label{tmp:1stZxy}
        \abs{Z(x,y)} \leq C \intM[\Bal{x}{3/\sqa}]{\frac{1}{\dg{x,z}^{n-\beta}}\frac{1}{\dg{z,y}^{n-\gamma}}}(z)\\ + \abs{\intM[M\setminus\Bal{x}{3/\sqa}]{X(x,z)Y(z,y)}(z)},
    \end{multline}
    this comes from fact that for $z\in \Bal{x}{3/\sqa}$,
    \begin{align*}
        \abs{X(x,z)} &\leq C\a^p \dg{x,z}^\rho \expab{x,y} \leq C \dg{x,z}^{\beta-n} & &\text{when }\dg{x,z} \geq 1/\sqa\\
        \abs{Y(z,y)} &\leq C\a^q \dg{z,y}^\nu \expab{x,y} \leq C \dg{z,y}^{\gamma -n} & & \text{when }\dg{z,y} \geq 1/\sqa
    \end{align*}
    for a constant $C>0$ independent of $\a,x,y,z$, by \eqref{eq:tmpcondpro}.
    
    Write $r := \dg{x,y} \leq 2/\sqa$, we first claim that 
    \begin{multline}\label{tmp:1inZxy}
        \intM[\Bal{x}{3/\sqa}]{\frac{1}{\dg{x,z}^{n-\beta}}\frac{1}{\dg{z,y}^{n-\gamma}}}(z)\\ \leq C\begin{cases}
            \dg{x,y}^{-(n-\beta-\gamma)} & \text{when }\beta+ \gamma < n\\
            1 + \abs{\log(\sqa \dg{x,y})} & \text{when }\beta + \gamma = n\\
            \a^{-\frac{\beta+\gamma-n}{2}}& \text{when }\beta+ \gamma > n
        \end{cases}.
    \end{multline}
    To prove \eqref{tmp:1inZxy}, we decompose $\Bal{x}{3/\sqa}$ in three parts: $\Bal{x}{r/2}$, $\Bal{x}{3r/2}\setminus\Bal{x}{r/2}$ and $\Bal{x}{3/\sqa}\setminus\Bal{x}{3r/2}$. First, when $ z\in \Bal{x}{r/2}$, we have $\dg{y,z} \geq r/2$, so that
    \begin{equation}\label{tmp:est1}
    \begin{aligned}  
        \intM[\Bal{x}{r/2}]{\frac{1}{\dg{x,z}^{n-\beta}}\frac{1}{\dg{z,y}^{n-\gamma}}}(z) &\leq C r^{\gamma-n}\intM[\Bal{x}{r/2}]{\frac{1}{\dg{x,z}^{n-\beta}}}(z)\\
            &\leq C r^{\gamma+\beta -n}
        \end{aligned}
    \end{equation}
    since $r/2 \leq 1/\sqa < \inj$ and $\beta>0$. 
    Similarly, when $z\in \Bal{x}{3r/2} \setminus\Bal{x}{r/2}$, we have $\dg{x,z} \geq r/2$, so that
    \begin{multline}\label{tmp:est2}  
        \intM[\Bal{x}{3r/2} \setminus\Bal{x}{r/2}]{\frac{1}{\dg{x,z}^{n-\beta}}\frac{1}{\dg{z,y}^{n-\gamma}}}(z)\\
        \begin{aligned}
            &\leq C r^{\beta-n} \intM[\Bal{x}{3r/2} \setminus\Bal{x}{r/2}]{\frac{1}{\dg{y,z}^{n-\gamma}}}(z)\\
            &\leq C r^{\beta+\gamma-n}
        \end{aligned}
        \end{multline}
    since $5r/2 \leq 5/\sqa < \inj$ and $\gamma > 0$.
    Finally, when $z\in \Bal{x}{3/\sqa} \setminus\Bal{x}{3r/2}$, we have $\frac{1}{3}\dg{x,z} \leq \dg{z,y} \leq \frac{5}{3} \dg{x,z}$, so that 
    \begin{multline*}
        \intM[\Bal{x}{3/\sqa}\setminus\Bal{x}{3r/2}]{\frac{1}{\dg{x,z}^{n-\beta}}\frac{1}{\dg{z,y}^{n-\gamma}}}(z)\\
        \begin{aligned}
            &\leq C \intM[\Bal{x}{3/\sqa}\setminus\Bal{x}{3r/2}]{\frac{1}{\dg{x,z}^{2n-\beta-\gamma}}}\\
            &\leq C \begin{cases}
                r^{\beta+\gamma-n} & \text{when } \beta + \gamma < n\\
                1+ \abs{\log \sqa r} & \text{when } \beta + \gamma = n\\
                \a^{-\frac{\beta+\gamma-n}{2}} & \text{when } \beta + \gamma > n
            \end{cases}
        \end{aligned} 
    \end{multline*}
    since $3/\sqa < \inj$. This concludes the proof of \eqref{tmp:1inZxy} when $r\leq 2\sqa$, realizing that when $\beta + \gamma > n$, $r^{\beta + \gamma - n} \leq C\a^{-\frac{\beta + \gamma -n}{2}}$ in \eqref{tmp:est1} and \eqref{tmp:est2}.
    \par We now claim that 
    \begin{equation}\label{tmp:2inZxy}
        \intM[M\setminus \Bal{x}{3/\sqa}]{\abs{X(x,z)}\abs{Y(z,y)}}(z) \leq C\begin{cases}
            r^{\beta + \gamma -n} & \text{when } \beta + \gamma < n\\
            \a^{-\frac{\beta+\gamma -n}{2}} & \text{when }\beta + \gamma \geq n
        \end{cases}.
    \end{equation}
    To prove \eqref{tmp:2inZxy}, note that, by assumption on $X,Y$, the integral in \eqref{tmp:2inZxy} has non-zero contribution only on the support of $X$ and $Y$, i.e. on $\Bal{x}{\tau} \cap \Bal{y}{\sigma}$. Moreover, when $z\in M\setminus \Bal{x}{3/\sqa}$, we have $\dg{x,z} \geq 1/\sqa$, $\dg{z,y} \geq 1/\sqa$ and $\frac{1}{3}\dg{x,z} \leq \dg{z,y} \leq \frac{5}{3} \dg{x,z}$, so that
    \begin{multline*}
        \intM[M\setminus \Bal{x}{3/\sqa}]{\abs{X(x,z)}\abs{Y(z,y)}}(z) \\
        \begin{aligned}
            &\leq C \intM[\bpr{\Bal{x}{\tau} \cap \Bal{y}{\sigma}}\setminus\Bal{x}{3/\sqa}]{\a^{p+q} \dg{x,z}^{\rho+\nu} e^{-\frac{4}{3}\sqa \dg{x,z}}}(z)\\
            &\leq C \a^{p+q} \int_{3/\sqa}^{\tau}t^{\rho + \nu + n -1} e^{-\frac{4}{3}\sqa t} dt \leq C \a^{-\frac{\beta+\gamma-n}{2}},
        \end{aligned}
    \end{multline*}
    where the last inequality follows from \eqref{eq:tmpcondpro}. This concludes the proof of \eqref{tmp:2inZxy} for $r \leq 2/\sqa$, realizing that when $\beta + \gamma < n$, $\a^{-\frac{\beta + \gamma -n}{2}} \leq C r^{\beta + \gamma - n}$. 
    \par Combining \eqref{tmp:1inZxy} and \eqref{tmp:2inZxy} with \eqref{tmp:1stZxy}, we have proven that, for $\dg{x,y} \leq 2/\sqa$,
    \begin{equation}\label{tmp:Zxy}
        \abs{Z(x,y)} \leq \begin{cases}
            C \dg{x,y}^{-(n-\beta-\gamma)} & \beta+ \gamma < n\\
            C (1 + \abs{\log(\sqa \dg{x,y})}) & \beta + \gamma = n\\
            C \a^{-\frac{\beta+\gamma-n}{2}}& \beta+ \gamma > n
        \end{cases}.
    \end{equation}

    \par For the second part of the proof, assume now that $\dg{x,y} \geq 2/\sqa$. Write again $r := \dg{x,y}$, we split the domain $M$ in the integral that defines $Z(x,y)$ in several parts: $\Bal{x}{1/\sqa}$, $\Bal{y}{1/\sqa}$, $\Bal{x}{3r/2} \setminus(\Bal{x}{1/\sqa} \cup \Bal{y}{1/\sqa})$ and $M \setminus \Bal{x}{3r/2}$. As before, the integral has non-zero contributions only in $\Omega_0 := \Bal{x}{\tau} \cap \Bal{y}{\sigma}$, and $\overline{\Omega}_0 \subset \Bal{x}{\inj},\, \overline{\Omega}_0 \subset \Bal{y}{\inj}$.
    It is then clear that $Z(x,\cdot)$ is supported in $\Bal{x}{\tau + \sigma}$. Without loss of generality, we can therefore assume that $\Omega_0 \neq \emptyset$, i.e. we can restrict to the case where $\dg{x,y} \leq \sigma + \tau$.
    \par For $z \in \Bal{x}{1/\sqa}$, we have $\dg{y,z} \geq r - 1/\sqa $ and $\frac{1}{2}\dg{x,y} \leq \dg{y,z} \leq \frac{3}{2} \dg{x,y}$, so that
    \begin{align*}
        \abs{\intM[\Bal{x}{1/\sqa}]{X(x,z)Y(z,y)}(z)} &\leq C \intM[\Bal{x}{1/\sqa}]{\a^q \frac{\dg{y,z}^\nu}{\dg{x,z}^{n-\beta}} \expab{z,y}}\\
            &\leq C \a^q r^\nu e^{-\sqa r} \intM[\Bal{x}{1/\sqa}]{\frac{1}{\dg{x,z}^{n-\beta}}}(z)\\
            &\leq C \a^{q-\frac{\beta}{2}} r^{\nu} e^{-\sqa r}\\
            &\leq C \a^{n -\frac{\beta+\gamma}{2} + \frac{\rho + \nu}{2}} r^{\rho + \nu +n} e^{-\sqa r}
    \end{align*}
    since $1/\sqa < \inj$, where the last inequality follows from \eqref{eq:tmpcondpro} and $n + \rho \geq 0$. The same arguments on $\Bal{y}{1/\sqa}$ similarly show that
    \[  \abs{\intM[\Bal{y}{1/\sqa}]{X(x,z)Y(z,y)}(z)} \leq C \a^{n -\frac{\beta+\gamma}{2} + \frac{\rho + \nu}{2}} r^{\rho + \nu +n} e^{-\sqa r}.
        \]
    Now for $ z \in \Bal{x}{3r/2}\setminus \bpr{\Bal{x}{1/\sqa}\cup \Bal{y}{1/\sqa}}$, we have $\dg{x,z} + \dg{z,y} \geq \dg{x,y}$, so that 
    \begin{multline}\label{eq:tmpbiggestterm}
        \abs{\intM[\Bal{x}{3r/2} \setminus\bpr{\Bal{x}{1/\sqa}\cup \Bal{y}{1/\sqa}}]{X(x,z)Y(z,y)}(z)}\\
            \leq C \a^{p+q} e^{-\sqa r} \intM[\Omega_0 \cap \Bal{x}{3r/2} \setminus\bpr{\Bal{x}{1/\sqa}\cup \Bal{y}{1/\sqa}}]{\dg{x,z}^\rho \dg{z,y}^\nu}(z).
    \end{multline}
    We claim that 
    \begin{equation}\label{tmp:drodnu}
        \intM[\Omega_0 \cap \Bal{x}{3r/2} \setminus\bpr{\Bal{x}{1/\sqa}\cup \Bal{y}{1/\sqa}}]{\dg{x,z}^\rho \dg{z,y}^\nu}(z) \leq C r^{\rho + \nu + n}
    \end{equation}
    for $r \geq 2/\sqa$.
    To see this, we decompose the domain of integration in two parts: $\Bal{x}{r/2} \setminus \Bal{x}{1/\sqa}$ and $\Omega_0 \cap \Bal{x}{3r/2} \setminus \bpr{\Bal{x}{r/2}\cup \Bal{y}{1/\sqa}}$. When $z\in \Bal{x}{r/2}$, we have $r/2 \leq \dg{z,y} \leq 3r/2$, so that
    \[  \intM[\Bal{x}{r/2} \setminus \Bal{x}{1/\sqa}]{\dg{x,z}^\rho \dg{z,y}^\nu}(z) \leq C r^{\rho + \nu + n},
        \]
    since $r/2 < \inj$ and $\rho + n >0$. For analogous reasons, we have 
    \[  \intM[\Omega_0 \cap \Bal{x}{3r/2} \setminus \bpr{\Bal{x}{r/2}\cup\Bal{y}{1/\sqa}}]{\dg{x,z}^\rho \dg{z,y}^\nu}(z) \leq C r^{\rho + \nu + n}.
        \]
    This concludes the proof of \eqref{tmp:drodnu}. Putting \eqref{eq:tmpbiggestterm} with \eqref{tmp:drodnu}, we have proven that, for $r \geq 2/\sqa$, 
    \begin{multline*}
        \abs{\intM[\Bal{x}{3r/2} \setminus\bpr{\Bal{x}{1/\sqa}\cup \Bal{y}{1/\sqa}}]{X(x,z)Y(z,y)}(z)}\\
        \leq C \a^{n -\frac{\beta+\gamma}{2} + \frac{\rho + \nu}{2}} r^{\rho+\nu+n}e^{-\sqa r}
    \end{multline*}
    using \eqref{eq:tmpcondpro}.
        
    Finally, when $z \in M \setminus \Bal{x}{3r/2}$, we have as before $\frac{1}{3} \dg{x,z} \leq \dg{y,z} \leq \frac{5}{3} \dg{x,z}$, so that since $X(x, \cdot)$ is supported in $\Bal{x}{\tau}$,
    \begin{multline*}
        \abs{\intM[M\setminus\Bal{x}{3r/2}]{X(x,z)Y(z,y)}(z)} \\
        \begin{aligned}
            &\leq C\a^{p+q} \intM[\Bal{x}{\tau} \setminus \Bal{x}{3r/2}]{\dg{x,z}^{\rho + \nu} e^{-\sqa \frac{4}{3}\dg{x,z}}}(z)\\
            &\leq C \a^{p+q}\, \Gamma(\rho+\nu +n, 2\sqa r)
        \end{aligned}
    \end{multline*}
    where $\Gamma(\delta, t) := \int_t^{+\infty} s^{\delta-1}e^{-s}ds$ is the incomplete Gamma function. Note that this last integral is non-zero only in the case where $r < 2\tau/3$. It is easily seen that $\Gamma(\delta, t) \sim t^\delta e^{-t}$ as $t\to \infty$, so that in the end we have shown that, for $r \geq 2/\sqa$,
    \[
        \abs{\intM[M\setminus\Bal{x}{3r/2}]{X(x,z)Y(z,y)}(z)} \leq \a^{n -\frac{\beta+\gamma}{2} + \frac{\rho + \nu}{2}} r^{\rho + \nu + n} e^{-\sqa r}
    \]
    using \eqref{eq:tmpcondpro}. This concludes the second part of the proof for $\dg{x,y} \geq 2/\sqa$.
    \par Finally, when $1 \leq \sqa \dg{x,y} \leq 2$, as before, the two regimes coincide, up to a constant. The first part of the proof shows that, with \eqref{tmp:Zxy}, 
    \[
        \abs{\intM{X(x,z)Y(z,y)}} \leq C \a^{\frac{n-\beta-\gamma}{2}}.
        \]
    Moreover, there is a constant $C>0$ independent of $\a$, $x,y$, such that
    \[  \tfrac{1}{C} \a^{\frac{n-\beta - \gamma}{2}} \leq \a^{n - \frac{\beta + \gamma}{2} + \frac{\rho + \nu}{2}} \dg{x,y}^{\rho + \nu + n} e^{-\sqa \dg{x,y}} \leq C \a^{\frac{n-\beta-\gamma}{2}}.
        \]
    We can then conclude, when $1 \leq \sqa \dg{x,y} \leq 2$ we have
    \[  \abs{Z(x,y)} \leq C \a^{n - \frac{\beta + \gamma}{2} + \frac{\rho + \nu}{2}} \dg{x,y}^{\rho + \nu + n} e^{-\sqa \dg{x,y}}.
        \]
    \par Regarding the continuity, fix $x,y \in M$ with $x\neq y$, and take any sequence $((x_m,y_m))_{m}$ such that $x_m \to x$, $y_m \to y$. Let $\delta_m := \dg{x_m,x},\, \Tilde{\delta}_m := \dg{y_m,y}$, we assume without loss of generality that $\dg{x_m,x} \leq \tfrac{1}{3}\dg{x, y}$ and $\dg{y_m,y} \leq \tfrac{1}{3}\dg{x,y}$ for all $m\in \N$. Then we compute
    \begin{multline}\label{eq:contZ}
        Z(x_m,y_m) = \intM[\Bal{x_m}{\delta_m/2}]{X(x_m,z)Y(z,y_m)}(z)\\
             +\intM[\Bal{y_m}{\Tilde{\delta}_m/2}]{X(x_m,z)Y(z,y_m)}(z)\\
                 + \intM[M\setminus (\Bal{x_m}{\delta_m/2}\cup \Bal{y_m}{\Tilde{\delta}/2})]{X(x_m,z)Y(z,y_m)}(z).
        \end{multline} 
    On the one hand, when $z\in \Bal{x_m}{\delta_m/2}$, we have $\dg{z,y_m} > r/2$, writing once again $r = \dg{x,y}>0$, so that 
    \begin{align*}  
        \abs{\intM[\Bal{x_m}{\delta_m/2}]{X(x_m,z)Y(z,y_m)}(z)} &\leq Cr^{\gamma-n}\intM[\Bal{x_m}{\delta_m/2}]{\dg{x_m,z}^{\beta-n}}(z)\\
            &\leq C \delta_m^\beta r^{\gamma-n}.
        \end{align*}
    Similarly, when $z \in \Bal{y_m}{\Tilde{\delta}_m/2}$, we have $\dg{x_m,z} > r/2$, so that 
    \begin{align*}  
        \abs{\intM[\Bal{y_m}{\Tilde{\delta}_m/2}]{X(x_m,z)Y(z,y_m)}(z)} &\leq Cr^{\beta-n}\intM[\Bal{y_m}{\Tilde{\delta}_m/2}]{\dg{z,y_m}^{\gamma-n}}(z)\\
            &\leq C \Tilde{\delta}_m^\gamma r^{\beta-n}.
        \end{align*}
    On the other hand, when $z\not\in (\Bal{x_m}{\delta_m/2} \cup \Bal{x_m}{\Tilde{\delta}_m/2})$, we have $\dg{x_m,z}\geq \tfrac{1}{3}\dg{x,z}$ and $\dg{z,y_m} \geq \tfrac{1}{3}\dg{z,y}$, so that 
    \begin{align*}  
        \abs{X(x_m,z)} &\leq C \dg{x,z}^{\beta-n},\\
        \abs{Y(z,y_m)} &\leq C \dg{z,y}^{\gamma-n} 
        \end{align*}
    By dominated convergence, we obtain 
    \begin{multline*}  \intM[M\setminus(\Bal{x_m}{\delta_m/2}\cup \Bal{y_m}{\Tilde{\delta}_m/2})]{X(x_m,z)Y(z,y_m)}(z)\\ \xto{m\to \infty} \intM{X(x,z)Y(z,y)}(z).
        \end{multline*}
    Coming back to \eqref{eq:contZ}, we have shown that for all $x\neq y$ in $M$, 
    \[  \lim_{m\to \infty} Z(x_m,y_m) = \intM{X(x,z)Y(z,y)}(z) + 0 + 0= Z(x,y),
        \]
    and $Z$ is continuous at $(x,y) \in M\times M\setminus \Diag$.
    Additionally, when $\beta+\gamma > n$, for all $w \in M$, and for all $0<\delta< 1/\sqa$, there exists $C>0$ such that for all $x,y \in M$,
    \[  \abs{\intM[\Bal{w}{\delta}]{X(x,z)Y(z,y)}(z)} \leq C \delta^{\beta+\gamma-n},
        \] 
    this holds true even when $x=y$. We conclude that $Z \in C^0(M\times M)$.
\end{proof}
\begin{remark}
    The assumption \eqref{eq:tmpcondpro} is a compatibility condition. If $2p - \rho = n-\beta$, we know that the two regimes of $X$ are equivalent when $\dg{x,y} \sim 1/\sqa$.
\end{remark}
\begin{remark}\label{rem:convGiraud}
    Observe that the convolution $Z$ decreases less quickly than either $X$ and $Y$ at large distances. This is due to the term $\dg{x,y}^{\rho+\nu+n}$ that comes from \eqref{eq:tmpbiggestterm}. This term becomes larger than $\a^{-\frac{\rho+\nu+n}{2}}$ when $\dg{x,y} \gg 1/\sqa $, which happens at finite distance as $\a \to \infty$.
\end{remark}

We conclude this Appendix with variant of Lemma \ref{prop:rexpGir}, which is used in the proof of Theorem \ref{prop:bvorGga}, where $X$ and $Y$ are allowed to have slightly different exponential decay.

\begin{lemma}\label{prop:girPsi}
    Let $\a_0\geq 1$ and let $X, Y \in C^0(M\times M\setminus \Diag)$ be such that $X(x,\cdot)$ is supported in $\Bal{x}{\tau}$, $\tau < \inj/2$. Suppose that there are $p\geq 0$ and $\rho > -n$ with $2p-\rho \leq 0$, and $0<\epsilon<1$ such that for all $x\neq y$,
    \begin{align*}
        \abs{X(x,y)} &\leq \begin{cases}
            1 & \sqa \dg{x,y}\leq 1\\
            \a^p \dg{x,y}^\rho \expab{x,y} & \sqa \dg{x,y} \geq 1
        \end{cases}\\
        \abs{Y(x,y)} & \leq \begin{cases}
            1 & \sqa \dg{x,y} \leq 1\\
            e^{-(1-\epsilon)\sqa \dg{x,y}} & \sqa \dg{x,y} \geq 1\\
            e^{-(1-\epsilon)\sqa \,\inj/2} & \dg{x,y} \geq \inj/2.
        \end{cases}
    \end{align*}
    Let, for all $x, y \in M$, $Z(x,y) := \intM{X(x,z)Y(z,y)}(z)$. There exists $\a_0\geq 1$ and $C>0$ such that for all $x, y \in M$ and $\a \geq \a_0$,
    \[  \abs{Z(x,y)} \leq C \a^{-\frac{n}{2}} \begin{cases}
        1 & \sqa \dg{x,y} \leq 1\\
        e^{-(1-\epsilon)\sqa \dg{x,y}} & \sqa \dg{x,y} \geq 1\\
        e^{-(1-\epsilon)\sqa \,\inj/2} & \dg{x,y} \geq \inj/2
    \end{cases}.
        \]
\end{lemma}
\begin{remark}
    Classically, the exponential decay of $Z$ exactly matches that of the least decreasing term $Y$.
\end{remark}

The proof follows the same steps as the proof of Lemma \ref{prop:rexpGir}, and we just explain the modifications. 
\begin{proof}
    \par Assume first that $\dg{x,y} \leq 2/\sqa $. Arguing as in the proof of Lemma \ref{prop:rexpGir} in the case $\beta = \gamma = n$, we obtain
    \[  \abs{Z(x,y)} \leq C \a^{-\frac{n}{2}}.
        \]
    Assume now that $2/\sqa \leq \dg{x,y} \leq \inj/2$. We adapt the proof of Lemma \ref{prop:rexpGir}. Write $r=\dg{x,y}$ and split the domain $M$ in the integral that defines $Z(x,y)$ between $\Bal{x}{1/\sqa}$, $\Bal{y}{1/\sqa}$, and $M\setminus(\Bal{x}{1/\sqa}\cup\Bal{y}{1/\sqa})$.
    \par When $z\in \Bal{x}{1/\sqa}$, we have $\dg{y,z} \geq r -1/\sqa$, so that
    \begin{align*}
        \abs{\intM[\Bal{x}{1/\sqa}]{X(x,z)Y(z,y)}(z)} &\leq C e^{-(1-\epsilon)\sqa r} \intM[\Bal{x}{1/\sqa}]{}(z)\\
            &\leq C \a^{-\frac{n}{2}}e^{-(1-\epsilon)\sqa r}.
    \end{align*}
    Now when $z \in \Bal{y}{1/\sqa}$, we have 
    \[  \abs{X(x,z)} \leq C e^{-(1-\epsilon)\sqa r},
        \]
    since $\dg{x,z} \geq r -1/\sqa$ and $p-\rho/2 < 0$. Therefore, we obtain
    \begin{align*}
        \abs{\intM[\Bal{y}{1/\sqa}]{X(x,z)Y(z,y)}(z)} &\leq C e^{-(1-\epsilon)\sqa r} \intM[\Bal{y}{1/\sqa}]{}(z)\\
            &\leq C \a^{-\frac{n}{2}}e^{-(1-\epsilon)\sqa r}.
    \end{align*}
    Finally, when $z \in M\setminus(\Bal{x}{1/\sqa}\cup\Bal{y}{1/\sqa})$, we have $\dg{x,z} + \dg{z,y} \geq r$ and $\dg{x,z} + \inj/2 \geq \inj/2 \geq r$, so that since $X(x,\cdot)$ is supported in $\Bal{x}{\tau}$,
    \begin{multline}\label{tmp:Xsupp}  
        \abs{\intM[M \setminus (\Bal{x}{1/\sqa}\cup\Bal{y}{1/\sqa})]{X(x,z)Y(z,y)}(z)}\\
        \begin{aligned}
            &\leq C \a^p e^{-(1-\epsilon)\sqa r} \intM[\Bal{x}{\tau} \setminus \Bal{x}{1/\sqa}]{\dg{x,z}^{\rho}e^{-\epsilon \sqa \dg{x,z}}}(z)\\
            &\leq C\a^{p-\frac{\rho}{2}-\frac{n}{2}} e^{-(1-\epsilon)\sqa r} \int_{\epsilon}^\infty t^{\rho + n - 1} e^{-t} dt\\
            &\leq C\a^{-\frac{n}{2}}e^{-(1-\epsilon)\sqa r}
        \end{aligned}
    \end{multline}
    since $p-\rho/2 < 0$. This concludes the proof for the case $2/\sqa \leq \dg{x,y} \leq \inj/2$. 
    \par For the last case, assume that $\dg{x,y} \geq \inj/2$, and split the domain $M$ in the integral that defines $Z(x,y)$ in two parts: $\Bal{x}{1/\sqa}$ and $M\setminus\Bal{x}{1/\sqa}$. When $z\in \Bal{x}{1/\sqa}$, we have $\dg{z,y} \geq \dg{x,y} - 1/\sqa \geq \inj/2 - 1/\sqa$, so that 
    \begin{align*}  
        \abs{\intM[\Bal{x}{1/\sqa}]{X(x,z)Y(z,y)}(z)} &\leq C e^{-(1-\epsilon)\sqa \inj/2} \intM[\Bal{x}{1/\sqa}]{}(z)\\
            &\leq C \a^{-\frac{n}{2}} e^{-(1-\epsilon)\sqa\inj/2}.
    \end{align*}
    On the other hand, when $z\in M\setminus \Bal{x}{1/\sqa}$, we have 
    \[  \dg{x,z} + \dg{z,y} \geq \dg{x,y} \geq \inj/2, 
    \] 
    so that, as in \eqref{tmp:Xsupp},
    \begin{multline*}
        \abs{\intM[M \setminus \Bal{x}{1/\sqa}]{X(x,z)Y(z,y)}(z)}\\
        \begin{aligned}    
            &\leq C \a^p e^{-(1-\epsilon)\sqa \inj/2} \intM[\Bal{x}{\tau}\setminus \Bal{x}{1/\sqrt{\a}}]{\dg{x,z}^\rho e^{-\epsilon\sqa\dg{x,z}}}(z)\\
            &\leq C \a^{-\frac{n}{2}} e^{-(1-\epsilon)\sqa\inj/2}
        \end{aligned}
    \end{multline*}
    using $p-\rho/2 < 0$, which concludes the proof.
\end{proof}

\bibliographystyle{amsplain}
\bibliography{reference}

\providecommand{\bysame}{\leavevmode\hbox to3em{\hrulefill}\thinspace}
\providecommand{\MR}{\relax\ifhmode\unskip\space\fi MR }
\providecommand{\MRhref}[2]{%
  \href{http://www.ams.org/mathscinet-getitem?mr=#1}{#2}
}
\providecommand{\href}[2]{#2}
\begin{thebibliography}{10}

\bibitem{AbrSteg}
Milton Abramowitz and Irene~A. Stegun, \emph{Handbook of mathematical functions
  with formulas, graphs, and mathematical tables}, National Bureau of Standards
  Applied Mathematics Series, vol. No. 55, U. S. Government Printing Office,
  Washington, DC, 1964.

\bibitem{AdPaYa95}
Adimurthi, Filomena Pacella, and S.~L. Yadava, \emph{Characterization of
  concentration points and {$L^\infty$}-estimates for solutions of a semilinear
  {N}eumann problem involving the critical {S}obolev exponent}, Differential
  Integral Equations \textbf{8} (1995), no.~1, 41--68. \MR{1296109}

\bibitem{Aubin82}
Thierry Aubin, \emph{Nonlinear analysis on manifolds. {M}onge-{A}mp\`ere
  equations}, Grundlehren der mathematischen Wissenschaften [Fundamental
  Principles of Mathematical Sciences], vol. 252, Springer-Verlag, New York,
  1982.

\bibitem{Car24}
Lorenzo Carletti, \emph{Attaining the optimal constant for higher-order
  {S}obolev inequalities on manifolds via asymptotic analysis}, arXiv preprint
  2408.09234, 2024.

\bibitem{DacqMeiSw05}
Anna Dall'Acqua, Christian Meister, and Guido Sweers, \emph{Separating
  positivity and regularity for fourth order {D}irichlet problems in
  2d-domains}, Analysis (Munich) \textbf{25} (2005), no.~3, 205--261.
  \MR{2232852}

\bibitem{Duf15}
Dean~G. Duffy, \emph{Green's functions with applications}, second ed., Advances
  in Applied Mathematics, CRC Press, Boca Raton, FL, 2015. \MR{3379916}

\bibitem{FefGra12}
Charles Fefferman and C.~Robin Graham, \emph{The ambient metric}, Annals of
  Mathematics Studies, vol. 178, Princeton University Press, Princeton, NJ,
  2012. \MR{2858236}

\bibitem{FelHebRob05}
Veronica Felli, Emmanuel Hebey, and Fr\'{e}d\'{e}ric Robert, \emph{Fourth order
  equations of critical {S}obolev growth. {E}nergy function and solutions of
  bounded energy in the conformally flat case}, NoDEA Nonlinear Differential
  Equations Appl. \textbf{12} (2005), no.~2, 171--213. \MR{2184079}

\bibitem{GazGruSw10}
Filippo Gazzola, Hans-Christoph Grunau, and Guido Sweers, \emph{Polyharmonic
  boundary value problems}, Lecture Notes in Mathematics, vol. 1991,
  Springer-Verlag, Berlin, 2010, Positivity preserving and nonlinear higher
  order elliptic equations in bounded domains.

\bibitem{GilTrud}
David Gilbarg and Neil~S. Trudinger, \emph{Elliptic partial differential
  equations of second order}, Classics in Mathematics, Springer-Verlag, Berlin,
  2001, Reprint of the 1998 edition.

\bibitem{Gir29}
Georges Giraud, \emph{Sur le probl\`eme de {D}irichlet g\'{e}n\'{e}ralis\'{e}
  (deuxi\`eme m\'{e}moire)}, Ann. Sci. \'{E}cole Norm. Sup. (3) \textbf{46}
  (1929), 131--245. \MR{1509295}

\bibitem{GJMS92}
C.~Robin Graham, Ralph Jenne, Lionel~J. Mason, and George A.~J. Sparling,
  \emph{Conformally invariant powers of the {L}aplacian. {I}. {E}xistence}, J.
  London Math. Soc. (2) \textbf{46} (1992), no.~3, 557--565.

\bibitem{Gru21}
Hans-Christoph Grunau, \emph{Optimal estimates from below for {G}reen functions
  of higher order elliptic operators with variable leading coefficients}, Arch.
  Math. (Basel) \textbf{117} (2021), no.~1, 95--104.

\bibitem{GruRob10}
Hans-Christoph Grunau and Fr\'{e}d\'{e}ric Robert, \emph{Positivity and almost
  positivity of biharmonic {G}reen's functions under {D}irichlet boundary
  conditions}, Arch. Ration. Mech. Anal. \textbf{195} (2010), no.~3, 865--898.

\bibitem{GruSw96}
Hans-Christoph Grunau and Guido Sweers, \emph{Positivity for perturbations of
  polyharmonic operators with {D}irichlet boundary conditions in two
  dimensions}, Math. Nachr. \textbf{179} (1996), 89--102.

\bibitem{Heb03}
Emmanuel Hebey, \emph{Sharp {S}obolev inequalities of second order}, J. Geom.
  Anal. \textbf{13} (2003), no.~1, 145--162.

\bibitem{Heb14}
\bysame, \emph{Compactness and stability for nonlinear elliptic equations},
  Zurich Lectures in Advanced Mathematics, European Mathematical Society (EMS),
  Z\"{u}rich, 2014.

\bibitem{HebVau96}
Emmanuel Hebey and Michel Vaugon, \emph{Meilleures constantes dans le
  th\'{e}or\`eme d'inclusion de {S}obolev}, Ann. Inst. H. Poincar\'{e} C Anal.
  Non Lin\'{e}aire \textbf{13} (1996), no.~1, 57--93. \MR{1373472}

\bibitem{HumRau09}
Emmanuel Humbert and Simon Raulot, \emph{Positive mass theorem for the
  {P}aneitz-{B}ranson operator}, Calc. Var. Partial Differential Equations
  \textbf{36} (2009), no.~4, 525--531. \MR{2558328}

\bibitem{MazVet}
Saikat Mazumdar and Jerome Vetois, \emph{Existence results for the higher-order
  q-curvature equation}, arXiv preprint 2007.10180, 2022.

\bibitem{Mic10}
Benoît Michel, \emph{Masse des opérateurs gjms}, arXiv preprint 1012.4414,
  2010.

\bibitem{Rey02}
Olivier Rey, \emph{The question of interior blow-up-points for an elliptic
  {N}eumann problem: the critical case}, J. Math. Pures Appl. (9) \textbf{81}
  (2002), no.~7, 655--696. \MR{1968337}

\bibitem{Rob10}
Frédéric Robert, \emph{Existence et asymptotiques optimales des fonctions de
  green des opérateurs elliptiques d'ordre deux.},
  https://iecl.univ-lorraine.fr/files/2021/04/ConstrucGreen.pdf (Accessed
  November 22, 2023), 2010.

\bibitem{Sch84}
Richard Schoen, \emph{Conformal deformation of a {R}iemannian metric to
  constant scalar curvature}, J. Differential Geom. \textbf{20} (1984), no.~2,
  479--495. \MR{788292}

\bibitem{SchYau81}
Richard Schoen and Shing~Tung Yau, \emph{Proof of the positive mass theorem.
  {II}}, Comm. Math. Phys. \textbf{79} (1981), no.~2, 231--260. \MR{612249}

\bibitem{Wan95}
Zhi~Qiang Wang, \emph{High-energy and multi-peaked solutions for a nonlinear
  {N}eumann problem with critical exponents}, Proc. Roy. Soc. Edinburgh Sect. A
  \textbf{125} (1995), no.~5, 1003--1029. \MR{1361630}

\end{thebibliography}

\end{document}